\documentclass[10pt,onecolumn,twoside]{IEEEtran}
\usepackage{amsmath,amssymb,bbm,amsthm}

\usepackage{paralist}
\usepackage{graphics} 
\usepackage{epsfig} 
\usepackage{epstopdf}

\usepackage{color}
\usepackage{setspace}
\usepackage{mathdots,mathrsfs}
\usepackage{latexsym,amsfonts,cite,comment,relsize}
\usepackage{stmaryrd}
\usepackage{caption}
\usepackage[dvips]{psfrag}
\usepackage{setspace}
\usepackage{url}
\usepackage{soul}
\usepackage{pgf,tikz} 
\usepackage{graphicx}
\usepackage{subfig}

\newtheorem{thm}{Theorem}

\newtheorem{lem}[thm]{Lemma}

\newtheorem{assum}[thm]{Assumption}

\theoremstyle{definition}

\newtheorem{rem}[thm]{Remark}

\title{ Convergence Analysis of Classes of Asymmetric Networks of Cucker-Smale Type \\ with Deterministic Perturbations}
\author{Christoforos Somarakis$^{1}$, Evripidis Paraskevas$^{2}$, John S. Baras $^{2}$ and  Nader Motee $^{1}$  %<-this % stops a space
\thanks{$^{1}$C. Somarakis and N. Motee are with the Department of Mechanical Engineering, Lehigh University,
        Bethlehem, PA, 18015
        {\tt\small \{csomarak,motee\}@lehigh.edu}}%
\thanks{$^{2}$E. Paraskevas and J. S. Baras are with the Department of Electrical Engineering, University of Maryland,
        College Park, MD 20740, USA
        {\tt\small \{evripar,baras\}@umd.edu}}%
}

\begin{document}

\maketitle

\begin{abstract}
We introduce and discuss two nonlinear perturbed extensions of the Cucker-Smale model with asymmetric coupling weights. The first model assumes a finite collection of autonomous agents aiming to perform a consensus process in the presence of identical internal dynamics. The second model describes a similar population of agents that perform velocity alignment with the restriction of collision-free orbits. Although qualitatively different, we explain how these two non-trivial types of perturbations are analyzed under a unified framework. Rigorous analysis is conducted towards establishing new sufficient conditions for asymptotic flocking to a synchronized motion.  Applications of our results are compared with simulations to illustrate the effectiveness of our theoretical estimates.
\end{abstract}

\section{INTRODUCTION}
For the past two decades, there has been a broad interest in the study of co-operative dynamic algorithms that run among a finite number of interconnected entities. Perhaps the most prominent family is this of consensus networks. The standard setting regards a finite number of agents $n$, where  state of each agent $z_i \in \mathbb R$ for $i=1,\dots,n$ is governed by the following protocol \begin{equation}\label{eq: linear}\begin{split} 
\dot z_i~=~\sum_{j=1}^{n} w_{ij}(t)(z_j -z_i ).
 \end{split}\end{equation} 
The nonnegative numbers $w_{ij}$ are coupling weights that characterize the effect of  state of agent $j$ onto the rate of change of the state of agent $i$. Certain criteria imposed on $w_{ij}$ ensure the desired asymptotic consensus property $$z_i(t)\rightarrow z^*~\text{as}~t\rightarrow \infty~\text{for}~i=1,\dots,n.$$ 
The common equilibrium point $z^*\in \mathbb R$ lies in the set defined by the convex hull of the initial values. The research on connectivity conditions for \eqref{eq: linear} is a saturated subject (see \cite{Jadbabaie_Lin_Morse_2003,Mesbahi_Egerstedt_2010,Moreau05,Olfati04,Tsitsiklis_Bertsekas_Athans_1986,6213081} and references therein). On the contrary, there are relatively few works in non-linear consensus dynamics. In the next subsection we review a few such extensions of \eqref{eq: linear}. 
\subsection{Types of Nonlinear Consensus Protocols}
In \cite{Moreau_2004} and \cite{MANFREDI201751} the authors study necessary and sufficient conditions of convergence to consensus in discrete and continuous time versions of non-linear networked co-operative systems represented by  
\begin{equation*}
\dot x_i(t+1)=f(t,x_1(t),\dots,x_n(t)\big)
\end{equation*} 
and 
\begin{equation*}
\dot x_i=f(t,x_1,\dots,x_n),
\end{equation*}respectively.
%
% In \cite{Krause00}, the authors propose a discrete version of the following nonlinear opinion dynamic model
%\begin{equation}\tag{1.1}\label{eq: non1}
%\dot z_i=\sum_{j=1}^nw_{ij}(z)(z_j-z_i),
%\end{equation} 
%where $z=[z_1, \ldots, z_n]^T$ and $w_{ij}(z)=w(|z_i-z_j|)$ models the rate of the opinion difference between agents.
In \cite{Arcak_2007,Hui20082375}, the authors investigate nonlinear versions of \eqref{eq: linear}  
\begin{equation}\tag{1.2}\label{eq: non2}
\dot z_i=\sum_{j=1}^n~w_{ij}(t,z_j-z_i)
\end{equation} 
and 
\begin{equation}\tag{1.3}\label{eq: non3}
\dot z_i=\sum_{j=1}^nw_{ij}(t,z_j)-\sum_{j=1}^nw_{ij}(t,z_i).
\end{equation} 
In network \eqref{eq: non2}, the analysis is based on passivity properties of the coupling functions $w_{ij}$. For the second model given by \eqref{eq: non3}, the investigation on the stability of the  consensus network is derived through the standard Lyapunov methods combined with LMI techniques.  In \cite{Cucker_Smale_2007,Cucker_Smale_2007_2},  the authors introduce a dynamical network that models the emergence of bird flocking. In their setting, the state of each agent $i$ consist of its vector position $x_i \in \mathbb R^{r}$ and velocity $v_i \in \mathbb R^{r}$, which is denoted by $(x_i,v_i) \in \mathbb R^{r}\times \mathbb R^{r}$. For an initial configuration $(x_i^0,v_i^0)$ for $i=1,\ldots,n$, the governing dynamics of the network are given by
\begin{equation}\label{eq: csflock}
\begin{split}
\dot x_i&=v_i \\
\dot v_i&=\sum_{j=1}^n~w_{ij}(x)(v_j-v_i).
\end{split}
\end{equation}
The class of dynamical networks \eqref{eq: csflock} are called as second-order flocking schemes. The simpler linear versions, where it is assumed that weights $w_{ij}$ are independent of the state, have  also appeared in \cite{Yu2010} and \cite{MARTIN201453}. In \eqref{eq: csflock}, the coupling weights are assumed to be decreasing functions of distance in the following form 
\begin{equation}\label{eq: csrates}w_{ij}(x)=\frac{K}{(\sigma^2+\|x_i-x_j\|^2)^{\beta}},
\end{equation} 
where $K, \sigma, \beta >0$ are coupling parameters and $\|\hspace{0.03cm}.\hspace{0.03cm}\|$ is the Euclidean norm.   
The explicit choice of $w_{ij}$  illustrates that the further the relative distance between two agents is, the less effect they have on each other. It may then occur that the agents will not be positioning themselves sufficiently close. This may in turn result the network to fail to preserve strong enough connectivity to achieve global speed alignment. The objective in \eqref{eq: csflock} is to derive initial conditions that ensure a strong enough network to allow global flocking to a common consensus state. The authors exploit the explicit form of $w_{ij}$, in particular its symmetry, and use algebraic graph theory to establish a stability condition that involves the initial data and the parameters $K,\sigma$ and $\beta$ from \eqref{eq: csrates}. Since the seminal work of Cucker and Smale, the network model \eqref{eq: csflock} has received considerable attention and it has been improved in several ways. In \cite{Ha_Liu_09}, a simple proof for general vanishing symmetric couplings is provided. In \cite{6787018} the authors derive convergence results for general symmetric weights without the infininite distance connectivity condition. In \cite{Motsch_Tadmo_14} extend the results for asymmetric couplings. In \cite{SomBarifac15} and \cite{SomBarecc15}, the authors provide extensions to asymmetric couplings with simple and switching connectivity in presence of time delay. Works \cite{Cucker_Mordecki_2008} and \cite{Cucker_2011} examine \eqref{eq: csflock} under the effect of additive noise and repelling symmetric functions, respectively.
% improvements have been published 
%but it is the version of \eqref{eq: csflock} that has in fact attracted immense interest over the years with several improvements.
 A general convergence result that summarizes the contribution of most of the aforemenioned works, illustrates the kind of sufficient conditions for global asymptotic velocity alignment while the network does not get dissolved. \begin{thm}[\cite{Motsch_Tadmo_14,SomBarecc15}]\label{thm: standardflocking}
Suppose that the nonlinear dynamical network  \eqref{eq: csflock} with coupling weights $w_{ij}(x)\geq \psi(\max_{ij}|x_i-x_j|)$, where $\psi$ is a non-negative integrable function. If the initial data satisfy $$\max_{i,j}|v_i^0-v_j^0|<\int_{\max_{i,j}|x_i^0-x_j^0|}^\infty \psi(s)\,ds,$$ then the solution $(x,v)$ of \eqref{eq: csflock} satisfies
\begin{description}
\item[(i)]  $|v_i(t)-v_j(t)|\rightarrow 0$  as  $t \rightarrow \infty$,
\item[(ii)] $\displaystyle \sup_{t \geq 0}\hspace{0.05cm}|x_i(t)-x_j(t)|<\infty$,~ for $i,j=1,\ldots,n$. 
\end{description}
\end{thm}
The technical argument behind the inequality condition in Theorem \ref{thm: standardflocking} involves a rate of convergence estimate of the state of the network to its equilibrium that explicitly depends on the network parameters. 
%Without this estimate,  it is impossible to conclude rigorous results for convergence. 
The appropriate theoretical machinery is provided through algebraic graph theory for symmetric couplings \cite{Chung_1997}  or through non-negative matrix theory \cite{Hartfiel,Seneta_2006}. All the above references focus on \eqref{eq: csflock} in exclusiveness. The stability conditions of Theorem \ref{thm: standardflocking} shed no light in the event that the agents are expected to collectively execute a more complex task different than convergence to a consensus state. 
  
\subsection{Our Contributions}

In this paper, we introduce and discuss a class of extensions of the classic Cucker-Smale model represented by \eqref{eq: csflock} that considers  nonlinear deterministic perturbations. These models describe how a group of agents can benefit from asymmetric state-dependent graphs with spatially-decaying couplings to perform more elaborate collaborative tasks, rather than performing the simple task of converging to a consensus state.

The importance of considering perturbed versions of \eqref{eq: csflock} is to understand the inherent interplay between being able to accomplish complex collective behaviors and the coupling structure of the underlying dynamical network.  Our main goal is to investigate to what extent we can push the envelope to modify existing gold-standard and well-studied dynamical networks to allow them to exhibit more complex collective behaviors. Towards this end, we address two types of state-dependent perturbations, relevant to network models of type \eqref{eq: csflock} that describe real-world paradigms.  %We develop a common analytic framework for convergence analysis in two particular scenarios of interest.

%{\BR {\bf Some rules to follow in all papers:} (1) Every sentence must carry information. (2) Avoid expressions and sentences that do not carry any information. (3) Never repeat something multiple times. State it once, but precisely and to the point. (4) Sentences should be lean; fatty sentences with extra useless wordings are considered redflags and they usually develop negative feelings in a reader towards the paper. (5) Always be honest with yourself while asking the following question from yourself: do I need this expression, sentence, or paragraph here ? If not, delete it. (6) Never leave anything for the reader to figure out; most of the times, it works against you. If necessary, explain your results clearly by not posing new puzzles for the reader to solve. }

The first extension takes into account the scenario of agents aiming to perform a consensus process in the presence of an internal generic dynamic rule. For instance, birds in a flock can fly individually or in coordination with other birds, possibly towards a synchronized complex motion different than a simple velocity co-ordination. This feature is modeled with an extra non-linear term in \eqref{eq: csflock} that makes the inequality flocking condition of Theorem \ref{thm: standardflocking} inapplicable. In Section \ref{subsec:synch}, it is shown how this class of networks can be analyzed. 

The second extension is a perturbation of \eqref{eq: csflock} that destabilizes the coordination process so long as the relative position of agents falls below a certain safety  margin. Such a safety restriction can prevent a network with fast decaying coupling strength from achieving global convergence. Similar to the first case, an extended version of Theorem \ref{thm: standardflocking} is needed to ensure flocking conditions. In Section \ref{subsec:colli}, we show that under a new set of technical conditions, design of flocking without collision can be achieved. 

% we explore the possibility of a collective consensus on the agents' velocity under the restriction that the agents relative distances must be persistently kept above a minimum value.

Although qualitatively different, we demonstrate how these two non-trivial types of perturbations can be analyzed under a unified framework. We explore the resulting dynamics and establish sufficient stability relations between the nonlinear coupling forces and the inherent dynamics, towards collective convergence within a connected network. Furthermore, we discuss extensions of our results for networks with non-linear coupling terms and explain the potential difficulties that can occur in the event of looser connectivity. Finally, several carefully tailored  examples are discussed to truthfully verify the strength of the imposed conditions.

%We rigorously investigate these interplays between coupling and internal (or repelling) dynamics and we provide sufficient conditions for convergence to an equilibrium synchronized state.
%   
%Our first proposed class of dynamical networks concerns group of agents with an inherent identical dynamic rule that governs the update of their velocity, in addition to the collective coupling. The convergence results relate the initial configuration the network coupling terms and the inherent dynamics. The asymptotic equilibrium may not be a mere constant point but a solution of the synchronized system.
%  guarantees convergence around a common speed while the flock remains connected at all times and within at least a prescribed distance. 
%
%{\BC You should motivate why these model are important and why one should care about them. Also, you should not refer to your models as 'examples'. State your results in the context of network systems: what will be achieved by your results? What's new about your results? Why are they important?}
%{\RC Also, rewrite the following text coherently. You should state your own new results first and then compare it with the existing literature and say that no one else has looked at this problem before.}  

\subsection{Literature Review} The subject of synchronization is rather mature and well-documented in the context of linear networks, with static or varying communication topologies (see \cite{WIELAND20111068,gonzalez2004synchronization} for control theoretic perspective or \cite{Cha02,Ott02} for a physics approach approach and references therein).

In \cite{SomBarecc16} the authors analyze first order synchronization schemes with time-dependent topologies where the internal dynamics may destabilize the alignment process. The work establishes sufficient conditions between the internal rule and the coupling scheme, in order for asymptotic synchronization to occur. To the best of our knowledge, the literature lacks a study of synchronization networks where the state of the network is correlated to the strength of the coupling as such is the case in networks of Cucker-Smale type. 

Prior works on the subject of flocking with collision avoidance in Cucker-Smale type networks include \cite{Cucker_2011,SomBarccc15}. The authors in \cite{Cucker_2011} study a collision-avoidance problem (model \eqref{eq: model2} of our work) on a framework based on the symmetry assumption of coupling weights. Their results, however, seize to hold when the couplings and the repelling functions are not symmetric. The asymmetry on the coupling terms has been highlighted in the literature as a very important feature in real-world biological and natural networks \cite{baler08}. The present work is an outgrowth of \cite{Sombarmed16} in various directions. The synchronization scheme is considered under milder assumptions and multi-dimensional internal dynamics. This gives rise to stability conditions that generalize the ones in \cite{Sombarmed16}. For the collision free model we provide a detailed proof of the convergence result. 
 
%extension of the contraction coefficient to two types of second-order consensus-based systems. 

%

%The stability properties of both of these networks are rigorously analyzed with appropriate modification of the mathematics used for examining the averaging properties of the contraction coefficient.
%
%This enables us to consider milder connectivity assumptions that extend and generalize existing results \cite{Cucker_2011}.

\section{Preliminaries}

\noindent {\it Notations:} For any $k\in \mathbb N$, $[k]$ is the set consisting of the first $k$ natural numbers, i.e. $[k]:=\{1,\dots,k\}$. By $n<\infty$ we understand the number of autonomous agents. The communication scheme is represented by a weighted graph $G=([n],E)$ where $[n]$ is the set of nodes and $E=\big\{w_{ij}: i,j \in [n]\big\}$ is the set of weighted edges. %The matrix representation of $G$ is established through the adjacency matrix $A=[a_{ij}]$ the degree matrix $D=\text{Diag}[\sum_{j}a_{ij}]$ and the Laplacian matrix $$L=D-A,$$ where its properties play a significant role in our analysis.
We denote the weighted degree of node $i$ as $d_i=\sum_{j}w_{ij}$. The dynamics evolve in $\mathbb R^r$ for some $r\geq 1$ that is endowed with the inner product $\langle \cdot , \cdot \rangle$ and the Euclidean norm $\|\cdot\|$. Each agent $i\in [n]$ is characterized by the state $(x_i,v_i)\in \mathbb R^r\times \mathbb R^r$. Clearly $x_i=(x_i^{(1)},\dots,x_i^{(r)})$ and $v_i=(v_i^{(1)},\dots,v_i^{(r)})$ stand for the position and velocity of $i$, respectively. Compact representation of the overall network state include $x=(x_1,\dots,x_n)$ and $v=(v_1,\dots,v_n)$ for elements $x$ and $v$ in the augmented space $\mathbb R^{n r}$. %The alignment measure of $u\in \mathbb R^{m  n}$ is $$\Lambda(u)=\frac{1}{n}\sum_{i<j}\|v_i-v_j\|^2.$$ 
For $l\in [r]$, the spread of $y\in \mathbb R^{nr}$ in the $l$-th dimension is $$S_l (y)=\max_{i}y_i^{(l)}-\min_i y_i^{(l)}=\max_{i,j}|y_i^{(l)}-y_j^{(l)}|$$ and finally we denote $S(y)=\max_l S_l(y)$.  %The space of continuous functions defined in $U$ and taking values in $V$ with $s\geq  0$ continuous derivatives is defined as $C^s(S,V)$. 
Throughout this paper, we reserve the notation $``\cdot"$ for the classic derivative and $\frac{d}{dt}$ for the right-hand Dini derivative.

\vspace{0.2cm}
\noindent {\it The Contraction Coefficient:}  The problem of stability of general consensus networks is associated with the asymptotic behavior of products of non-negative matrices \cite{Seneta_2006}. The underlying mathematical concept used in the proofs of those results was first discussed in one of Markov's first papers \cite{Hartfiel}. This concept, known as the contraction coefficient (or coefficient of ergodicity), measures the worst case estimate on the averaging effect of non-negative matrices when they act on vectors. Markov demonstrated that for a $P=[p_{ij}]$ non-negative $n\times n$ matrix with constant row sums (i.e.  $\sum_{j}p_{ij}\equiv m$ for some $m > 0$), it holds that \begin{equation}\label{eq: cc} S(Pz)\leq \big(m-\min_{i,i'}\sum_{k}\{p_{ik},p_{i'k}\}\big) S(z), ~\forall ~z\in \mathbb R^n.\end{equation}
A proof of this result can be found in \cite{Hartfiel}. The contraction coefficient is the quantity $\big(m-\min_{i,i'}\sum_{k}\{p_{ik},p_{i'k}\}\big)$ and it provides an explicit estimate of the averaging property of $P$. While the framework is primarily compatible with discrete time dynamics, it can be adapted to a continuous time framework, \cite{Motsch_Tadmo_14,SomBarecc15,SomBarecc16}.  The contributions of this paper critically rely on auxiliary results (Lemma \ref{lem: difineq} and Lemma \ref{lem: difineq2}, below) the proof of which is an elaborate adaptation of the proof of the contraction estimate \eqref{eq: cc} \footnote{The intereted reader can directly observe the similarities between the contraction estimates derived in those two Lemma with \ref{eq: cc}}. These ideas provide us with with worst-case contraction estimates of the consensus term. Then combining ideas from \cite{Ha_Liu_09} and \cite{Cucker_2011} we are able to derive sufficient conditions for the desired asymptotic stability.

\section{Networks of Cucker-Smale Type with Deterministic Perturbations} \label{sect: problem}
In this section we present the dynamic algorithms with the accompanying hypotheses. The discussion is concluded with the main convergence results. 

Let $n$ be a number of agents and $i\in [n]$ with the state $(x_i,v_i)\in \mathbb R^{r}\times \mathbb R^{r}$. The models we consider in this work can be cast as the following system of equations:
\begin{equation}\label{eq: perturbcsflock}
\begin{split}
\dot x^{(l)}_i&=v^{(l)}_i\\
\dot v^{(l)}_i&=\sum_{j}\big(w_{ij}(t,x)+b^{(l)}_{ij}(t,x,v)\big)(v^{(l)}_{j}-v^{(l)}_i), t\geq t_0
\end{split}
\end{equation} 
for $l\in [r]$, $w_{ij}$ are the coupling terms and $b_{ij}:\mathbb R\times \mathbb R^{l}\times \mathbb R^{l}\rightarrow \mathbb R^{l}$ a non-trivial deterministic perturbation  between agents $j$ and $i$. This perturbation may, in general, depend on both time and state. We look at two meaningful scenarios of this general model, assuming corresponding special forms of $b_{ij}$. One form leads to flocking to a synchronization solution while the other form leads to collision-less velocity coordination.
%
%
%\begin{equation}\label{eq: perturbcsflock2}
%\begin{split}
%\dot x_i &= v_i\\ 
%\dot v_i &= \sum_{j=1}^n W_{ij}(t,x,v) (v_j-v_i)
%\end{split}
%\end{equation}
%where  the local coupling matrix for a given $i \in [n]$ is defined by  
%\[ W_{ij}(t,x,v)=w_{ij}(t, x) I+\mathrm{diag}\big(v_{ij}^{(1)}(t,x,v),\ldots, v_{ij}^{(r)}(t,x,v)\big)\]
%
%\subsection{Capturing Deterministic Modeling Uncertainties}
\subsection{Heterogeneous Synchronization Using Diffusion Processes}\label{subsec:synch} The first type of perturbation considers a situation where agents, in addition to the consensus averaging, attain an individual way of flying that is dictated by an internal dynamical behavior. For this case we take \eqref{eq: perturbcsflock} with $$b_{ij}^{(l)}(t,x,v)=\frac{g^{(l)}(t,v_i)}{v_j^{(l)}-v_i^{(l)}}.$$ Then for given $t_0\in \mathbb R$ we are led to the initial value problem:  \begin{equation}\label{eq: model1}
\begin{split}
&\dot x_i=v_i \\
&\dot v_i=g(t,v_i)+\sum_{j}w_{ij}(t, x)(v_j-v_i),\hspace{0.1in} t\geq t_0\\
&x_i(t_0)=x_i^0, ~ v_i(t_0)=v_i^0\in \mathbb R^{r},~\hspace{0.1in}\text{given}.
\end{split}
\end{equation} 

An agent's velocity is affected by both the state of the other nodes and by an inherent dynamical process. 
The state of the other nodes affects the agent at a rate that depends on time and on the position vector $x$ It is, therefore, far from clear that the condition of Theorem \ref{thm: standardflocking} ensures convergence. Our objective is to reveal the interplay between the coupling forces of the consensus network, the initial configuration and the potential instability induced by the internal dynamics through a new stability condition. 
We proceed by stating assumptions on the acceptable behavior of the internal dynamic system and conclude with a condition on the coupling functions.

Let for $t\geq t_0$ the initial value problem \begin{equation}\label{eq: nominal}\dot z=g(t,z), \hspace{0.2in} z(t_0)=z^0\in \mathbb R^r \end{equation} together with its solution $z=z(t,t_0,z^0)$ defined in a maximal interval $[t_0,T)$. The hypothesis below aims to establish the well-posedness of $z$.
\begin{assum}\label{assum: f} The function $g(t,z)$ is defined in $\mathbb R\times V$, where $V$ an open subset of $\mathbb R^r$. It is continuous in $t\in \mathbb R$ and attains continuous first derivative in $z\in V$.
%\noindent (2.) The solution $z=z(t,t_0,z^0),~t\geq t_0$ exists for all times. 
%\noindent (3.) There exist $\alpha,\beta$ such that for every $\gamma\in (0,\alpha)$, there is $t^*=t^*(\alpha,\beta)\geq t_0$ such that $$\|z^0\|\leq \gamma~\Rightarrow~\|z(t,t_0,z^0)\|\leq \beta,~\forall ~ t\geq t_0+t^*.$$ 
\end{assum}
 % and the second its existence in the large. %The final condition implies that  $z$ is uniformly ultimately bounded with bound $\beta$.

We impose the following smoothness and bound conditions on the coupling weights. 
\begin{assum}\label{assum: coupling} The following statements hold for the functions $w_{ij}(t,x)$:

%\noindent(1.) They are non-negative, uniformly upper bounded functions of $t$ and $x$, i.e. $$ 0\leq w_{ij}(t,x)\leq \bar{w} $$

\noindent(1.) They are continuous functions of $t$ and continuously differentiable functions of $x$.

\noindent (2.) They satisfy \begin{equation*}\bar{w}\geq \sup_{t\geq t_0} \sup_{x}\max_{i\neq j}w_{ij}(t,x)\geq  \inf_{t}w_{ij}(t,x)\geq \psi(S(x))\end{equation*} for $\bar{w}<\infty$ and $\psi \geq 0$ an integrable and non-increasing function.
\end{assum} 
The rates $w_{ij}$ are uniformly bounded from above by $\bar{w}$ but not bounded from below as $\psi(\cdot)$ is allowed to vanish. This means that Assumption \ref{assum: coupling} allows the spatially decaying property. In addition, it includes symmetric couplings either as in \eqref{eq: csrates} or the ones proposed in \cite{Ha_Liu_09}. Clearly, it includes non-symmetric coupling rates such as those proposed in \cite{Motsch_Tadmo_14}. For the statement of the first result we are in need of some additional notation. Let
\begin{equation*}
K(t,l,y,w)=\int_{0}^{1}\frac{\partial }{\partial z^{(l)}}g^{(l)}(t,q y+(1-q)w)\,dq+\\
+\sum_{h\neq l}\bigg|\int_0^{1}\frac{\partial}{\partial z^{(h)}}g^{(l)}(t,q y+(1-q)w)\,dq\bigg|
\end{equation*} and for the maximal solution $\big(x(t),v(t)\big),~t\in [t_0,T)$ of \eqref{eq: model1}:
\begin{equation}\label{eq: K}
K=\sup_{t\in [t_0,T)}\max_{i,i'\in [n],l\in [r]} K(t,l,v_i(t),v_{i'}(t)).
\end{equation} 
This quantity represents the effect of the internal dynamic rule $g$ in the coupling process. Evaluated on the maximal solution, $K$, is the worst case estimate that we must take into account in order to derive initial conditions that compensate for the potential instability that $g$ will induce in the system that will in turn weaken the coupling rate $w_{ij}$. Observe also in $K(t,l,y,w)$ that the cross terms $\frac{\partial }{\partial z^(h)}g^{(l)}$ for $h\neq l$, are added in absolute value, as opposed to $\frac{\partial }{\partial z^(l)}g^{(l)}$. This discrepancy is the result of the dimensionality problem. Seeking synchronization of $v_i$ in all $r$ dimensions, one must take into account the effect that the internal dynamics $g$ have in all $r$ dimensions. Due to lack of structure on $g$, we have no choice but to regard the rate at which $g$ varies in different dimensions ($h\neq l$) as a purely negative perturbation against the synchronization along the $l^{th}$ dimensions.

%Set $$\tilde g=\sup_{t\geq t_0}\max_{l}\max_{y_a,y_b\in \mathcal U_l}\int_0^1 g_l'(t,q y_a+(1-q)y_b)\,dq $$ where $\mathcal U_l=[\min_i v_i^{(l)}(t_0),\max_i v_i^{(l)}(t_0)]$ and $g_l'=\frac{\partial g_l(t,z)}{\partial z}$ that in view of Assumption \ref{assum: f} is bounded. 

\begin{thm}\label{thm: main1} Consider the initial value problem \eqref{eq: model1} with Assumptions \ref{assum: f} and \ref{assum: coupling} to hold  and its maximal solution $(x(t),v(t)),~t\in [t_0,T)$. Assume also that

\noindent{(1.)} there exists $d>0$ such that $$S(v^0)<\int_{S(x^0)}^{d}\big(n\psi(r)-K\big)\,dr,$$

\noindent{(2.)} There exist $\varepsilon>0$ such that $$n\psi(d^*)\geq K+\varepsilon $$ for $d^{*}>0:S(v^0)=\int_{S(x^0)}^{d^{*}}\big(n\psi(r)-K\big)\,dr$ and $K$ as in \eqref{eq: K}. Then, $(x,v)$ satisfies $$S(v(t))\leq e^{-\varepsilon t}S(v^0) \hspace{0.1in}\text{\&} \hspace{0.1in}\sup_{t\geq t_0}S(x(t))<d^*,~\forall~t\in [t_0,T).$$
\end{thm}

The proof of the Theorem relies on the following result, the proof of which is put in the Appendix:
\begin{lem}\label{lem: difineq}
Let $(x, v)$ be the maximal solution of \eqref{eq: model1}, defined in $[t_0,T)$. Then
\begin{equation}\label{eq: difineq}
\begin{split}
\frac{d}{dt}S\big(v(t)\big)\leq \big[K-n\psi\big(S(x(t))\big)\big]S\big(v(t)\big),
\end{split}
\end{equation} for $t\in [t_0,T)$.
\end{lem}

\begin{proof}[Proof of Theorem \ref{thm: main1}] Consider the functional \begin{equation}
\begin{split}
\mathcal V( x, v)=S(v)+\int_{0}^{S( x)}\big(n\psi(r)-K\big)\,dr
\end{split}
\end{equation} and evaluate it at the solution $\big( x(t), v(t)\big), t\in[t_0,T)$ with $\mathcal V(t)=\mathcal V\big( x(t), v(t)\big)$. From (1.) there exists $t_1>t_0$ such that for $t\in [t_0,t_1)$ \begin{equation*}\begin{split}&\frac{d}{dt}\mathcal V(t)\leq \frac{d}{dt}S(v(t))+\big[n\psi\big(S(x(t))-K\big]S\big(v(t)\big)\leq 0
\end{split}\end{equation*} in view of Lemma \ref{lem: difineq} and $$\frac{d}{dt}S(x(t))\leq S(v(t)).$$
The last claim is justified as follows: Note that $S(x(t))=|x_i^{(l)}(t)-x_j^{(l)}(t)|$, for some $i,j\in [n]$ and $l\in [r]$ (possibly dependent on $t$). Then
\begin{equation*}\begin{split}
\frac{d}{dt}S(x(t))&=\frac{d}{dt}|x_i^{(l)}(t)-x_j^{(l)}(t)|\leq \bigg|\frac{d}{dt}\big(x_i^{(l)}(t)-x_j^{(l)}(t)\big)\bigg|\\
&= |v_i^{(l)}(t)-v_j^{(l)}(t)|\\
&\leq S(v(t))).
\end{split}
\end{equation*}
Consequently, $\mathcal V(t)\leq \mathcal V(t_0)$ for $t<t_1$, equivalent to
 \begin{equation*}
S( v(t))+\int_0^{S( x(t))}(n\psi(r)-K)\,dr\leq S(v^0)+\int_{0}^{S(x^0)}\big(n\psi(r)-K\big)\,dr
\end{equation*} and obviously $$ \int_0^{S(x(t))}(n\psi(r)-K)\,dr\leq S(v^0)+\int_{0}^{S(x^0)}\big(n\psi(r)-K\big)\,dr. $$ Condition $(1.)$ also implies the existence of $d^*<d$ as in condition $(2.)$ the inequality above yields
$$\int_0^{S(x(t))}(n\psi(r)-K)\,dr\leq \int_0^{d^*}(n\psi(r)-K)\,dr$$ so \begin{equation*}
\int_{S(x(t))}^{d^*}(n\psi(r)-K)\,dr\geq  0
\end{equation*} The last inequality implies that $S(x(t))\leq d^*$ for $t<t_1$ and $n\psi(d^*)-K\geq \varepsilon>0$. Since no assumption was taken on $t_1$, the monotonicity of $\psi$ yields that we can take $t_1=T$ proving the second claim of the theorem. The differential inequality \eqref{eq: difineq} then yields the first claim, concluding the proof.
\end{proof} This result establishes the connection between the internal and the position-dependent coupling dynamics. The power of Theorem \ref{thm: main1} can be further extracted if we assume that the solution $v$ is a-priori trapped within a region that possibly depends on the initial conditions. In such case, $T=\infty$ and the imposed conditions can be checked more easily. The following result asserts that for a particular type of compact subsets $\mathbb R^{r}$ into which $z$ of \eqref{eq: nominal} remains trapped, implies that $v$ in \eqref{eq: model1} behaves likewise. 
%
%a function of the solution $z$ and the synchronization network \eqref{eq: model1}.
%
%Let us examine a few cases of interest, the main results of which are stated as Corollaries.
\begin{thm}\label{cor: main1} 
Assume that $U$ be a compact, convex $g$-invariant subset of $\mathbb R^{r}$. Then  $v^0_i\in U$ for $i\in [n]$, guarantees that the solution $(x,v)$ of  \eqref{eq: model1} exists for all times. In addition, the results of Theorem \ref{thm: main1} hold true with $K$ as in \eqref{eq: K} to be substituted by $$K=\sup_{t\geq t_0}\max_{y,w\in U,l\in [r]}K(t,l,y,w).$$
\end{thm} 
\begin{proof}
Consider the maximal solution $(x(t),v(t))$, $t\in [t_0,T)$ of \eqref{eq: model1}. %The second equation that can be written as
%\begin{equation*}
%\dot v_i(t)=g(t,v_i(t))+\sum_j w_{ij}(t)\big(v_j(t)-v_i(t)\big)
%\end{equation*} where $w_{ij}(t)=w_{ij}(t,x(t)),~t\in [t_0,T)$ and $i\in [n]$. 
It suffices to show that $v_i(t_0)=v_i^0\in U$ implies $v_i(t)\in U$ for all $t\in [t_0,T)$ and $i\in [n]$. Then $T$ can be extended to $\infty$, establishing the solution $(x,v)$ in the large. We can discretize the second part of \eqref{eq: model1} as follows 
\begin{equation}\label{eq: conv1}
v_i(t+l)=v_i(t)+l\bigg[g(t,v_i(t))+\sum_j w_{ij}(t)\big(v_j(t)-v_i(t)\big)\bigg]
\end{equation} where $w_{ij}(t)=w_{ij}(t,x(t)),~t\in [t_0,T)$ and $i\in [n]$. The result is then proved if we can show that $v_i(t+l)\in U$ for $l$ small and arbitrary $t\in [t_0,T)$. The collection of the initial values $\{v_i^0\}$ lies in $U$, and the convexity of $U$ implies that the convex hull of $\{v_1^0,\dots,v_n^0\}$ is also a subset of $U$. Pick $i\in [n]$ and let's elaborate on $v_i(t+l) $ in Eq. \eqref{eq: conv1}, at $t=t_0$. By the $g$-invariance property of $U$, one can discretize the uncoupled equation and conclude that there is $k^{'}>0$ such that 
$$v_i^{a}(t_0+k):=v_i^0+k g(t_0,v_i^0)\in U$$
for any $k\in [0,k^{'})$. In addition, if on ignores the internal dynamics (i.e. $g\equiv 0$), Assumption \ref{assum: coupling} on the weights, implies that there is $k^{''}$ such that 
$$ v_i^{b}(t_0+k):=v_i^0+k\sum_{j}w_{ij}(t_0)(v_j^0-v_i^0) \in U$$
for $k\in [0,k^{''})$. The latter claim is because the right hand side of the equation is a convex sum of $\{v_i^0\}$, thus places $v_i^{b}(t_0+k)$ exactly in the convex hull of $\{v_1^0,\dots,v_n^0\}$, which in turn is in $U$. We showed that both $v_i^{a}(t_0+k)$ and $ v_i^{b}(t_0+k)$ are points in $U$ for small enough $k\in [0,\min\{k^{'},k^{''}\})$. Then by convexity of $U$, 
\begin{equation*}
\begin{split}
U \ni \frac{1}{2}v_i^{a}(t_0+k) + \frac{1}{2}v_i^{b}(t_0+k)&=\\ =v_i^0+\frac{k}{2}\bigg[ g(t_0,v_i^0)&+\sum_k w_{ij}(t_0)(v_j^0-v_i^0)\bigg] 
\end{split}
\end{equation*} Observe that the right hand-side is precisely this of \eqref{eq: conv1} with $l=k/2$ and $t=t_0$. We remark that in view of Assumptions \ref{assum: f} and \ref{assum: coupling}, $k^{'}$ and $k^{''}$ can be chosen independent of $t_0$ or $i$ so that the argumentation remains valid for any other $t\in [t_0,T)$ and $i\in [n]$.
\end{proof}

The applicability of this result includes, but not limited to, problems of synchronization of non-linear or chaotic oscillations where initial configuration near the synchronization manifold ensure boundedness of the solution of the network (see for example \cite{Ott02,Cha02,gonzalez2004synchronization}). 

\begin{rem}\label{rem: connectivity}
In Theorems \ref{thm: main1} and \ref{cor: main1} we assumed that every agent has access to the states of the rest of the agents. This is an effectively all-to-all communication and it may be thought of as too demanding. The invoked argument allows for a small relaxation of the all-to-all sharing: We can ask for every two agents that do not communicate, the existence of a third agent with which both of them must communicate. In this case, the estimates of Theorem \ref{thm: main1} are valid if $n\psi(r)$ in the initial setup conditions is substituted by $\psi(r)$.
\end{rem}

\subsection{Flocking with Guaranteed Collision Avoidance}\label{subsec:colli}
The second dynamic model we will study rolls back to the classic consensus problem and convergence to a common constant value.  
 Here the alignment should be achieved  with agents positioning themselves in at least a minimum distance from each other. For this we will study a special case of \eqref{eq: perturbcsflock} with $$b_{ij}^{(l)}(t,x,v)=-\frac{1}{S(v)}f_{ij}(\|x_{i}-x_j\|^2)\langle x_{i}-x_{j},v_{i}-v_{j}\rangle , ~ l\in [r].$$
The protocol we propose is \begin{equation}\label{eq: model2}
\begin{split}
&\dot x_i=v_i \\
&\dot v_i=\sum_{j}\bigg(w_{ij}(t,x)-\frac{f_{ij}(\|x_{i}-x_j\|^2)\langle x_{i}-x_{j},v_{i}-v_{j}\rangle}{S(v)}\bigg)(v_j-v_i)\\
&x_i(t_0)=x_i^0, ~ v_i(t_0)=v_i^0\in \mathbb R^{r},~\hspace{0.1in}\text{given}.
\end{split}
\end{equation} where $w_{ij}$ as in Assumption \ref{assum: coupling} and the additional terms that model the collision prevention mechanism. 
The functions $f_{ij}$ are repelling forces that can be appropriately constructed so as to keep the agents at a prescribed relative distance. 
\begin{assum}\label{assum: rep}
For any $i\neq j$, $f_{ij}(\cdot)$ is a continuous non-negative function defined in $(d_0,\infty)$, for some constant $d_0>0$ such that 
$$\int_{d_0}^{d_1}f_{ij}(r)\,dr=+\infty~\text{and}~ \int_{d_1}^{+\infty}f_{ij}(r)\,dr<+\infty,$$
for all $d_1>d_0$.
\end{assum}

A simple example of symmetric repelling function (also to be used in \S \ref{sect: simulation}) is $f(r)=(r-d_0)^{-\varepsilon}$ for any fixed $\varepsilon>1$. More examples can be found in \cite{Cucker_2011}. The objective of this section is the derivation of sufficient conditions for flocking of \eqref{assum: rep}. One should expect a formula that connects the coupling strength with the functions $f_{ij}$. 

\begin{thm}\label{thm: main2}
Consider the initial value problem \eqref{eq: model2} with Assumptions \ref{assum: coupling} and \ref{assum: rep} to hold, and its maximal solution $(x(t),v(t))$ for $t\in [t_0,T)$. If $i\neq j$ implies $\|x_i^0-x_j^0\|>d_0$ and \begin{equation}\label{eq: initialcolav}
\frac{S(v^0)}{n}<\frac{1}{2}\int_{S(x^0)}^{\infty}\psi(r)\,dr - \max_{i\neq j}\int_{\|x_i^0-x_j^0\|^2}^{\infty}f_{ij}(r)\,dr.
\end{equation}
Then:
\begin{enumerate}
\item $T=\infty$,
\item $\|x_i(t)-x_j(t)\|>d_0$, for all $t\geq t_0$,
\item the solution satisfies $$S(v(t))\rightarrow 0~\text{as}~t\rightarrow \infty\hspace{0.1in}\text{and} \hspace{0.1in}\sup_{t\geq t_0} S(x(t))<\infty.$$
\end{enumerate}
\end{thm}

The proof of Theorem \ref{thm: main2} follows partly the steps of the proof of Theorem \ref{thm: main1} and partly the arguments developed in \cite{Cucker_2011}. We begin with a preliminary result that establishes a solution estimate of $v$ similar to Lemma \ref{lem: difineq}. Both the preliminary and the main results rely on the following observation.
\begin{rem}\label{rem: h} Note that $b_{ij}(t,x,v)$ can be written as
\begin{equation}\label{eq: bij}
b_{ij}(t,x,v)=\frac{1}{2S(v)}\frac{d}{dt}\int_{\|x_{i}(t)-x_j(t)\|^2}^{\infty}f_{ij}(r)\,dr.
\end{equation} Given a solution $(x,v)$ of \eqref{eq: model2} defined in $t\in [t_0,T)$ we set $h=h_t$ and $h^{'}=h^{'}_t \in [n]$ the agents that lie closes to each other, i.e. the indices that minimize $||x_{i}-x_{i^{'}}||$. While this mapping may not be unique, it is a piece-wise constant function of $t$.
\end{rem}
\begin{lem}\label{lem: difineq2}
The maximal solution $(x,v)$ of \eqref{eq: model2} defined in $[t_0,T)$ satisfies for all $i,i^{'}\in [n]$ and $l \in [r]$
\begin{equation}\label{eq: difen2}
\frac{d}{dt}\big|v_i^{(l)}-v_{i'}^{(l)}\big| \leq - m \big|v_i^{(l)}-v_{i'}^{(l)}\big|+(m-\rho_{i,i^{'}})S(v)-\Gamma_{i,i^{'}},
\end{equation} where $m=m(t)$ is an arbitrary but fixed, non-negative, integrable function, $\rho_{i,i_{'}}=\sum_{j}\min\{w_{ij}(t,x(t)),w_{i^{'}j}(t,x(t))\}$ and 
\begin{equation*}
\begin{split}
\Gamma_{i,i^{'}}&=\frac{1}{2} \sum_{j\neq i} \min\bigg\{\frac{d}{dt}\int_{||x_i-x_j||^2}^{\infty}f_{ij}(r)\,dr,
 \frac{d}{dt}\int_{||x_{i^{'}}-x_j||^2}^{\infty}f_{i^{'}j}(r)\,dr\bigg\}
\end{split}
\end{equation*}.
\end{lem}
The proof of Lemma \ref{lem: difineq2} is put in the Appendix. When $i$ and $i^{'}$ are chosen so as to maximize $\max_{l} |v_i^{(l)}-v_{i^{'}}^{(l)}|=S(v)$ then one obtains the estimate
\begin{equation}\label{eq: maxdiam}
\frac{d}{dt}S(v)\leq -\rho_{i,i^{'}} S(v)-\Gamma_{i,i^{'}}.
\end{equation}
\begin{proof}[Proof of Theorem \ref{thm: main2}] For the reader's convenience we make the proof to consist of the following four steps: Collision avoidance conditions for the maximal solution, existence of the maximal solution for all times, persistent connectivity of the flock and velocity alignment. 

\noindent \textit{a) Collision Avoidance:} Define $\mathcal E:\mathbb R^{nr}\times \mathbb R^{nr}\rightarrow \mathbb R$ $$\mathcal E(x,v)=S(v)+|v_{h}^{(l)}-v_{h'}^{(l)}|+\sum_{j}\gamma_j^{i,i^{'}}+ \sum_j \gamma_{j}^{h,h^{'}}$$ where $l\in [r]$ is arbitrary but fixed,
\begin{equation*}\begin{split}
&\gamma_j^{i,i^{'}}=\gamma_j^{i,i^{'}}(x,v)=\int_{||x_k-x_j||^2}^{\infty}f_{kj}(r)\,dr~\text{with}~k=k_{i,i^{'}} :\\
&f_{kj}(||x_k-x_j||^2)\langle x_k-x_j,  v_k-v_j \rangle=\min\big\{ f_{ij}(||x_i-x_j||^2)\langle x_i-x_j,  v_i-v_j \rangle,  f_{i^{'}j}(||x_{i^{'}}-x_j||^2)\langle x_{i^{'}}-x_j,  v_{i^{'}}-v_j \rangle \big \}
\end{split}
\end{equation*}
We remind that $h,h^{'}\in [n]$ are as in Remark \ref{rem: h} and $i,i^{'}\in [n]$ are chosen to maximize the velocity diameter and consequently satisfy Eq. \eqref{eq: maxdiam}. All of these indices are potentially time dependent but piecewise constant. Differentiating $\mathcal E$ along $(x(t),v(t)),t\geq t_0$, Lemma \ref{lem: difineq2} for $m(t)=\rho_{i,i^{'}}(t,x(t))$ and \eqref{eq: maxdiam} will yield 
\begin{equation*}\begin{split}
\frac{d}{dt}\mathcal E(x(t),v(t))=&-\rho_{i,i^{'}}S(v)-\Gamma_{i,i^{'}}-\rho_{i,i^{'}}|v_{h}^{(l)}-v_{h^{'}}^{(l)}|+\big(\rho_{i,i^{'}}-\rho_{h,h^{'}}\big)S(v)-\Gamma_{h,h^{'}}\\
&+\frac{d}{dt}\sum_{j}\gamma_j^{i,i^{'}}(x(t),v(t))+\frac{d}{dt}\sum_{j}\gamma_j^{h,h^{'}}(x(t),v(t))\\
=&-\rho_{i,i^{'}}|v_{h}^{(l)}-v_{h^{'}}^{(l)}|-\rho_{h,h^{'}}S(v)\leq 0
\end{split}
\end{equation*} Note that $$\text{either}~\int_{\|x_{h}(t)-x_{h'}(t)\|^2}^{\infty}f_{hh^{'}}(r)\,dr~\text{or}~ \int_{\|x_{h}(t)-x_{h'}(t)\|^2}^{\infty}f_{h^{'}h}(r)\,dr$$ is a member of the sum $\sum_{j}\gamma_j^{h,h^{'}}$. Then $\mathcal E(t)\leq \mathcal E(t_0)$ implies that for the two agents $h$ and $h'$ that are in closest distance from each other
\begin{equation*}\begin{split}
&\min\bigg\{\int_{\|x_{h}(t)-x_{h'}(t)\|^2}^{\infty}f_{hh'}(r)\,dr, \int_{\|x_{h}(t)-x_{h'}(t)\|^2}^{\infty}f_{hh'}(r)\,dr\bigg\}  \leq \mathcal E(t)\leq \mathcal E(t_0) <\infty
\end{split}
\end{equation*} Since both $f_{hh^{'}}$ and $f_{h^{'}h}$ satisfy Assumption \ref{assum: rep}, we proved that $\|x_{i}(t)-x_{j}(t)\|\geq d^*$ for some $d^*>d_0$ for all $i\neq j\in [n]$ and $t\in[t_0 ,T)$.

\noindent \textit{b) Existence in the Large:} From $\mathcal E(t)\leq \mathcal E(t_0)$ we also deduce that $S(v(t))\leq \mathcal E(t_0)<\infty$ hence $$S(x(t))\leq S(x^0)+T\mathcal E(t_0):=\bar{X}$$ and the solution lies for $[t_0,T)$ in $$\Omega=\{(x,v): S(x)\leq \bar{X}, \|x_{i,j}\|\geq d^*,  i\neq j, S(v) \leq \mathcal E(t_0)\}  $$ where $d^*$ defined in the first part of the proof. However $\Omega$ is a compact subset of $$ \{(x,v): S(x)\leq \bar{X}, \|x_{i,j}\|> d_0,  i\neq j \}.$$ The fundamentals in the theory of differential equations assure, however, that this cannot occur if $T<\infty$ (see for example Theorem 1.21 in \cite{markley04}) hence it is  ensured that the solution is eventually defined for all $t\geq t_0$.

\noindent \textit{c) Bounded Flock:} We will show now that the initial settings we consider in the statement of the Theorem keeps all agents in finite relative  distance. We recall from the first step that 
\begin{equation*}
\frac{d}{dt}\mathcal E(x(t),v(t))\leq -\rho_{i,i^{'}}|v_{h}^{(l)}-v_{h^{'}}^{(l)}|-\rho_{h,h^{'}}S(v)\leq- \rho_{h,h^{'}}S(v)
\end{equation*} If we integrate from $t_0$ to $t$ and use the lower bound of $\rho_{h,h^{'}}$ taken in view of Assumption \ref{assum: coupling}  we obtain \begin{equation}\label{eq: bound on xu}
\begin{split}
&\mathcal E(t)-\mathcal E(t_0)\leq -\int_{t_0}^{t}n\psi\big(S(x(r))\big)S(v(r))\,dr\Rightarrow \\
&\mathcal E(t_0) \geq  \int_{t_0}^{t}n\psi\big(S(x(r))\big)S(v(r))\,dr.
\end{split}
\end{equation} If we set $p(s)=S(x(s))$ we observe that $\frac{dp}{ds}\leq S(v(s))$ and deduce $$\int_{S(x^0)}^{S(x(t))}n\psi(s)\,ds\leq \mathcal E(t_0).$$ Had the flock been dissolved, there should be a sequence $\{t_i\}_{i\geq 0}$ with $t_i\rightarrow \infty$ as $i\rightarrow \infty$, so that   $S(x(t_i))\rightarrow \infty$ as $t_i\rightarrow \infty$. This would mean that
\begin{equation*}\begin{split}
\int_{S(x^0)}^{\infty}n\psi(r)\,dr&\leq \mathcal E(t_0) \leq 2S(v^0)+2n\max_{i\neq j}\int_{||x_i^0-x_j^0||^2}^{\infty}f_{ij}(r)\,dr 
\end{split}
\end{equation*}
that is not possible in view of \eqref{eq: initialcolav}. Thus \begin{equation}\label{eq: boundflock}
\sup_{t\geq t_0}S(x(t))<\infty,
\end{equation} i.e. the flock remains connected and bounded.
\noindent \textit{d) Convergence to flocking:} At first we combine \eqref{eq: bound on xu} and \eqref{eq: boundflock} to conclude  $$\int^{\infty}S(v(r))\,dr<\infty.$$ Moreover, we used $\frac{d}{dt}\mathcal E(x(t),v(t))<0$ to show that $\sup_{t\geq t_0}S(v(t))<\mathcal E(t_0)$ and this implies that $|S(v(t))|$ is uniformly bounded. It only remains to show that $S(v(t))$ is uniformly continuous so that Barbalat's Lemma (Lemma 8.2 in \cite{Khalil}) applies to $k(t):=\int_{t_0}^{t}S(v(s))\,ds$ to conclude $$\dot{k}(t)=S(v(t))\rightarrow 0~\text{as}~t\rightarrow \infty.$$ The last condition is shown by direct application of the definition of uniform continuity: For any $t_1,t_2\geq t_0$ close to each other, there exist $i,i'\in [n]$ and $l\in[r]$ such that \begin{equation*}\begin{split}
|S(v(t_1))-S(v(t_2))|&=|(v_i^{(l)}(t_1)-v_{i'}^{(l)}(t_1))-(v_i^{(l)}(t_2)-v_{i'}^{(l)}(t_2))|\leq 2\max_{h\in\{i,i'\}}|v_h^{(l)}(t_1)-v_h^{(l)}(t_2)|\\
&\leq 2\max_{h\in\{i,i'\}} |\dot v_{h}^{(l)}(t^*)|\cdot |t_1-t_2|
\end{split}
\end{equation*}
Since $\dot v_{h}^{(l)}(t)$ satisfied \eqref{eq: model2}, its absolute value is bounded above by finite number of terms $|w_{ij}|\leq \bar{w}$, $$|b_{ij}(x,v)|\leq \max_{ij}f_{ij}(\underline{d}^2)\big(\sup_{t}S(x(t)))^2\mathcal E(t_0)$$ where $\underline{d}=\inf_{t}\min_{i\neq j}|x_{ij}(t)|>d_0$ and $|S(v(t))|\leq \mathcal E(t_0)$, each of which is independent of time and the uniform continuity property holds true. Barbalat's lemma can then be applied concluding the proof.
\end{proof}

\section{Examples And Simulations} \label{sect: simulation}
In this section we will briefly present applications of our rigorous results. We will use a group of $n=5$ and examine networks of different dimensions and types of solutions. The initial time is taken $t_0=0$. The coupling rates are set $w_{ij}(t,x)=w\frac{1.5+0.5\sin(t)}{(|x|+\beta_{ij}^2)^{\delta}}$ for $\delta\geq 0$, $\beta_{ij}\in (0,\sqrt{2})$ and $w\geq 0$ a uniform control parameter. All simulations are carried with the \texttt{ode23} routine in \texttt{MATLAB}. 
\subsection{Scalar Networks}
We consider \eqref{eq: nominal}  with $g(t,z)=\cos(t)(z-1)(z-2)$. Its solution $z(t)=z(t,0,z^0),~t\geq 0$ is $$z(t)=\frac{2-\frac{z_0-2}{z_0-1}e^{\sin(t)}}{1-\frac{z_0-2}{z_0-1}e^{\sin(t)}}.$$
 It can be easily checked that for $z^0\in (1,2)$ the solution exists for all times and is periodic with period 1. %For $z^0>2$, if the initial conditions $t_0$ and $z^0$ satisfy $\sin(t_0)+\ln((z_0-1)/(z_0-2))\in [-1,1]$, the solution blows up in finite time.
We consider now the network \eqref{eq: model1} and its maximal solution $(x,v)$ with  initial setup $v^0=(1.2,1.4,1.1,1.5,1.3)$. We select $w=1$ throughout this example. The solution $(x,v)$ exists for all times remaining trapped in $U=[1,2]$ but we can in fact say much more. Due to the monotonicity of solution $z$, using the same arguments as in the proof of Theorem \ref{cor: main1} we can calculate the estimate $K$ by evaluating it over $z(t,0,1.5)$, since no solution $v_i$ will essentially exceed $z(t,0,1.5)$. Thus we estimate $K$ as: 
\begin{equation*}
K\leq 2\frac{2+e^{\sin(t)}}{1+e^{\sin(t)}}-3\leq \frac{1-e^{-1}}{1+e^{-1}}\approx 0.462
\end{equation*}In addition, $w_{ij}(t,x)\geq \frac{1}{(|x|+2)^{\delta}}$. From the conditions of Theorem \ref{cor: main1}, exponential convergence to flocking with rate $\varepsilon>0$ can occur if we can find number $d$ such that 
\begin{equation*}
0.4<5\frac{(d+2)^{(1-\delta)}-(S(x^0)+2)^{(1-\delta)}}{1-\delta}-0.462(d-S(x_0))
\end{equation*} and \begin{equation*}
\frac{5}{(d^*+2)^{\delta}}-0.462>0
\end{equation*}
for $d^*$ that after substituting it to $d$, it can achieve equality in the first condition. At $\delta=0$ one can easily verify that the first condition holds true for every $d^*$ and that the first inequality is always satisfied for $d$ large enough. 
%Now we make use of the spread of $u^0$ and $x^0$ so that the first condition reads $$0<5\frac{(d+2)^{1-\delta}-3^{1-\delta}}{1-\delta}+0.062-0.462 d.$$ Note that at $\delta=1$ this becomes $0< $
As a brief numerical exploration we took $S(x^0)=1,3,5$ and examined the two conditions one after the other. The results are plotted in Figure \ref{fig: example1} and suggest that all the conditions can be satisfied for $\delta<1.1$. Figure, \ref{fig: example1_1} presents realizations of $(x,v)$ with $S(x^0)=5$. The first simulation depicts a network with the weak coupling $\delta=10$, where no synchronization can happen. In the second attempt we adopt a stronger coupling with $\delta=4$ that however still does not yield a feasible $d$ to satisfy our conditions. Finally for $\delta=0.9$ we can see synchronization of solutions that occur exponentially fast.
\begin{figure*}[h]
\begin{center}
\includegraphics[scale=0.4]{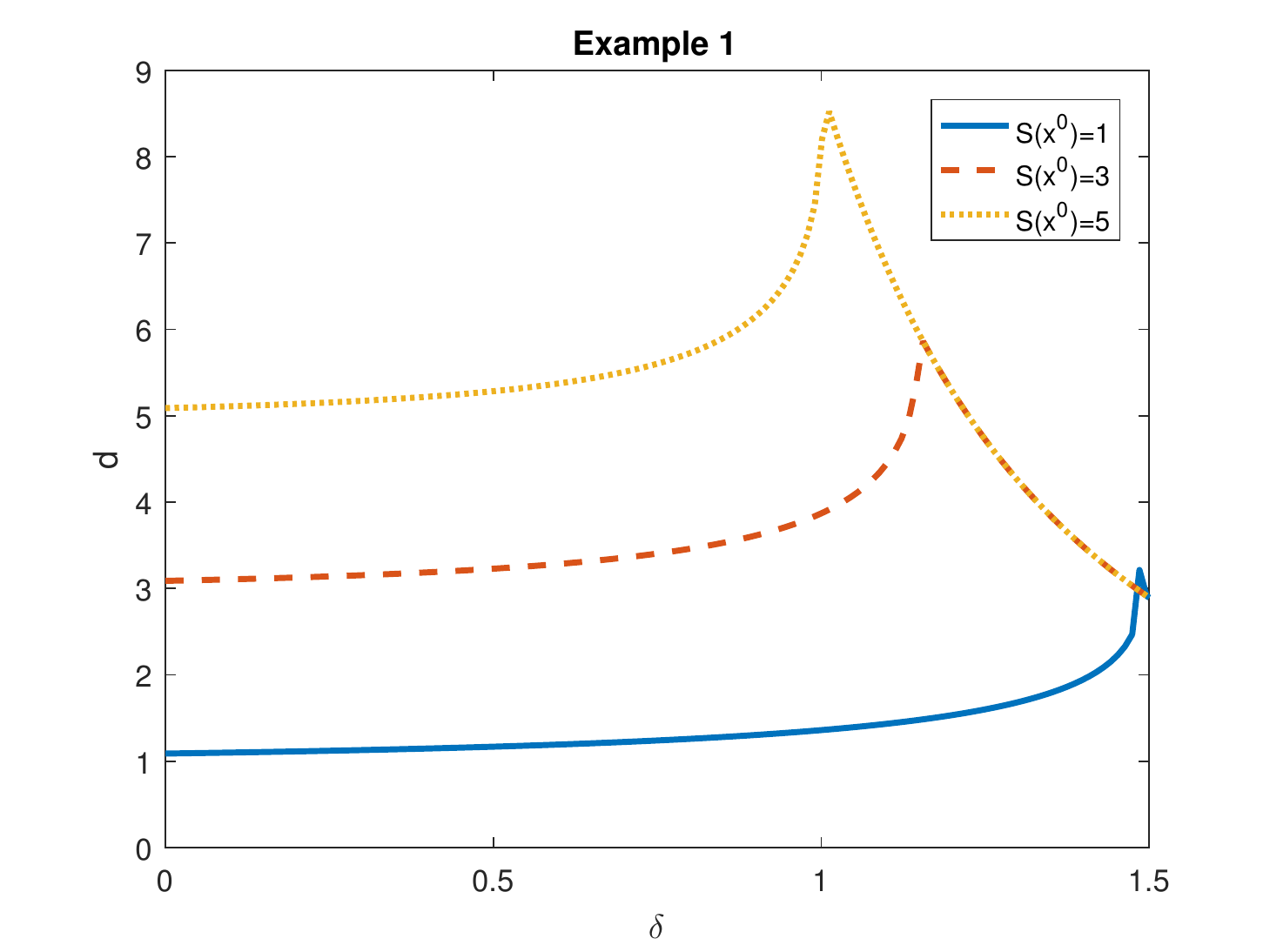}
\includegraphics[scale=0.4]{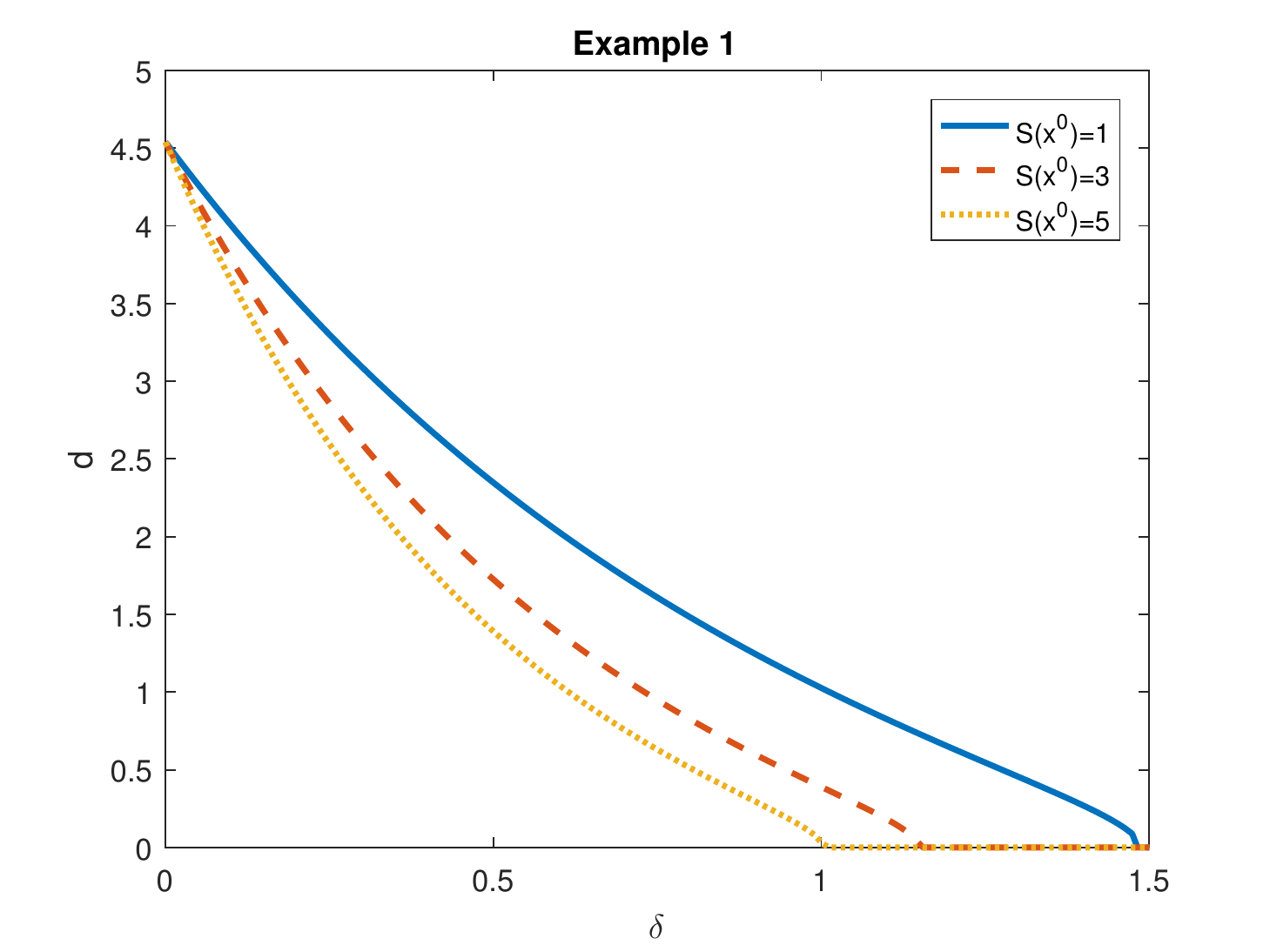}
\caption{The stability condition curves for different values of $S(x^0)$. The larger the initial spread of the relative positions, the weaker the couplings initially are. This makes the existence of a solution $d$ harder.}\label{fig: example1}
\end{center}
\end{figure*}

\begin{figure*}[h]
\begin{center}
\includegraphics[scale=0.4]{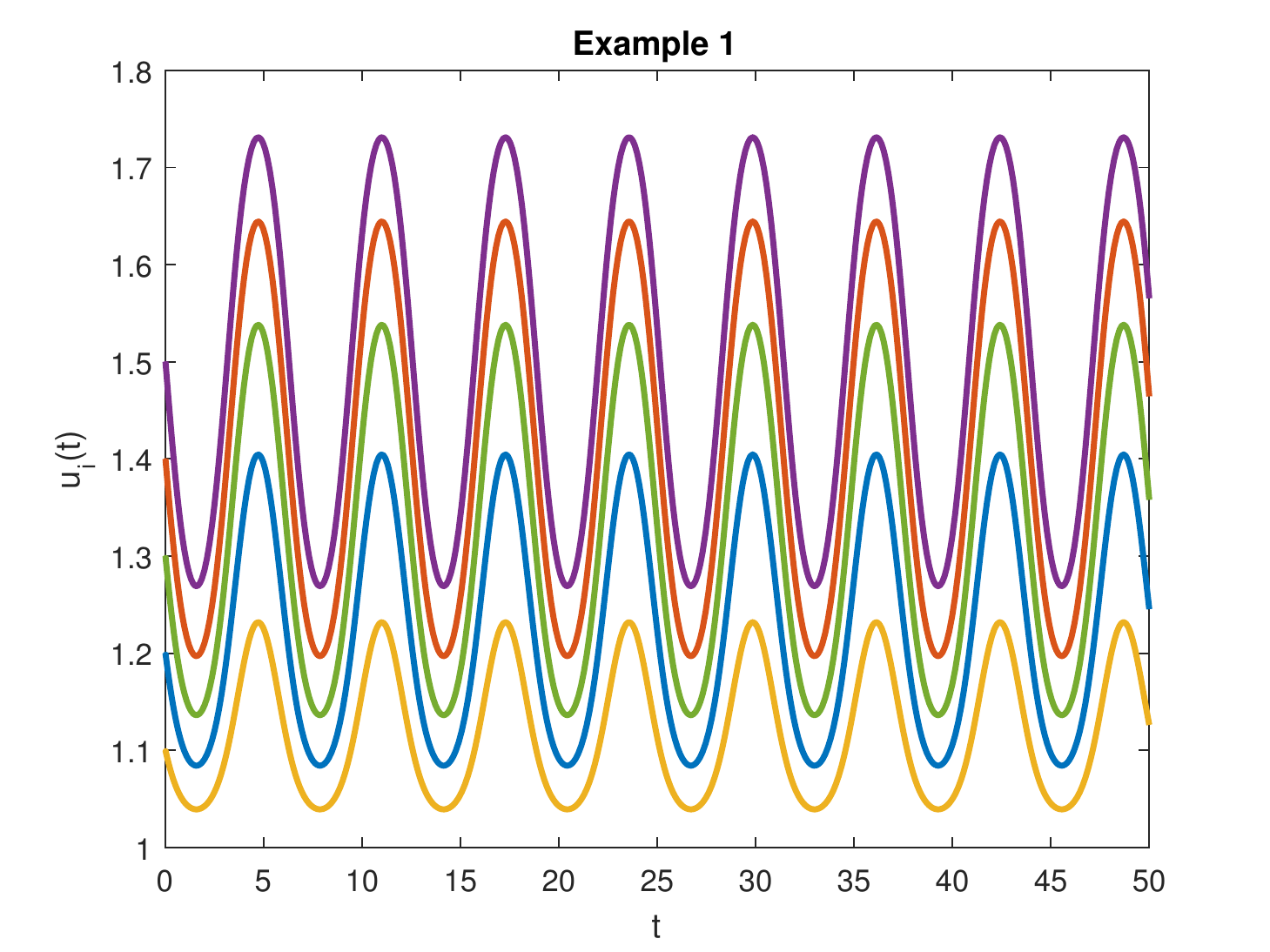}
\includegraphics[scale=0.4]{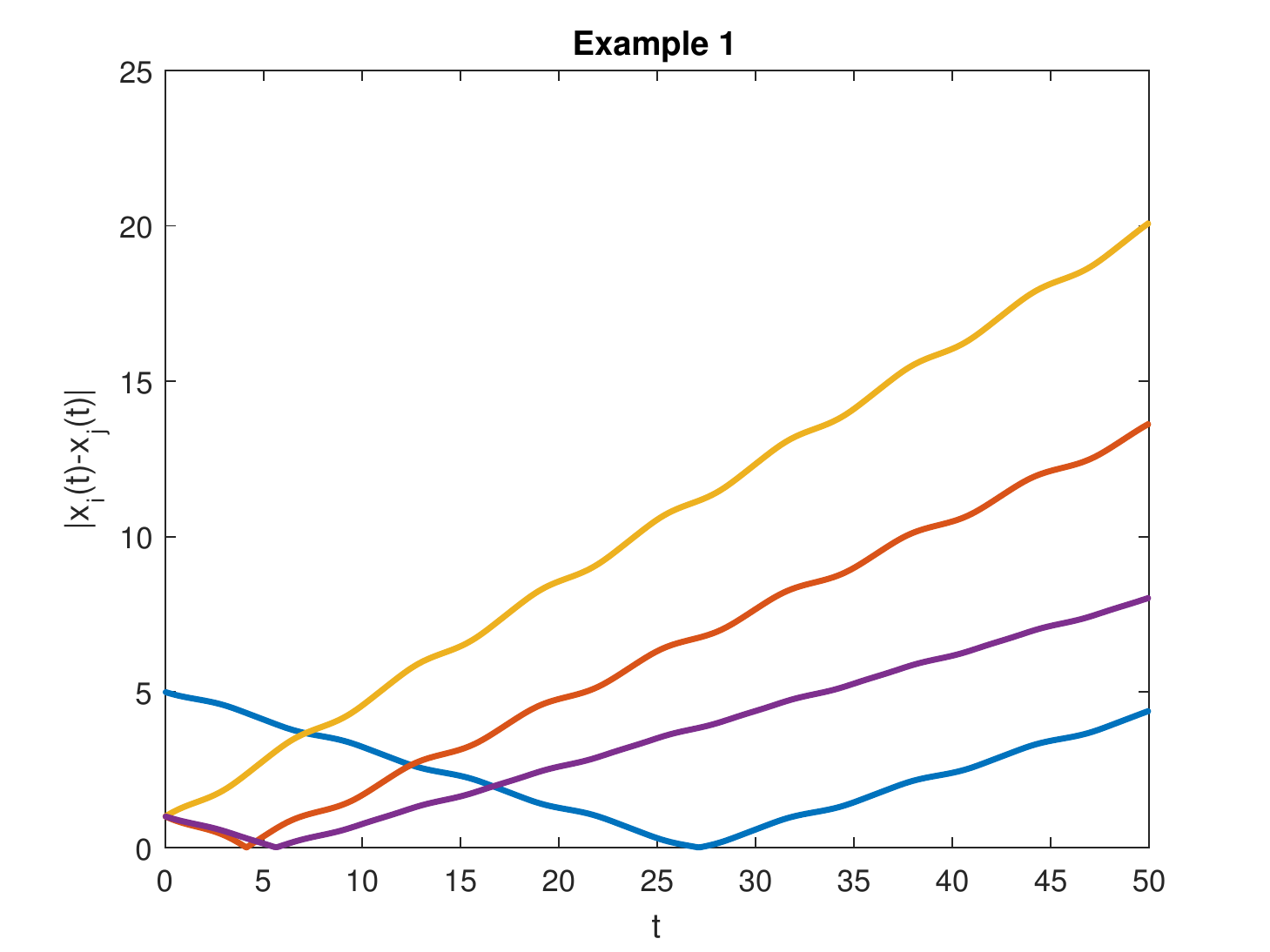}
\includegraphics[scale=0.4]{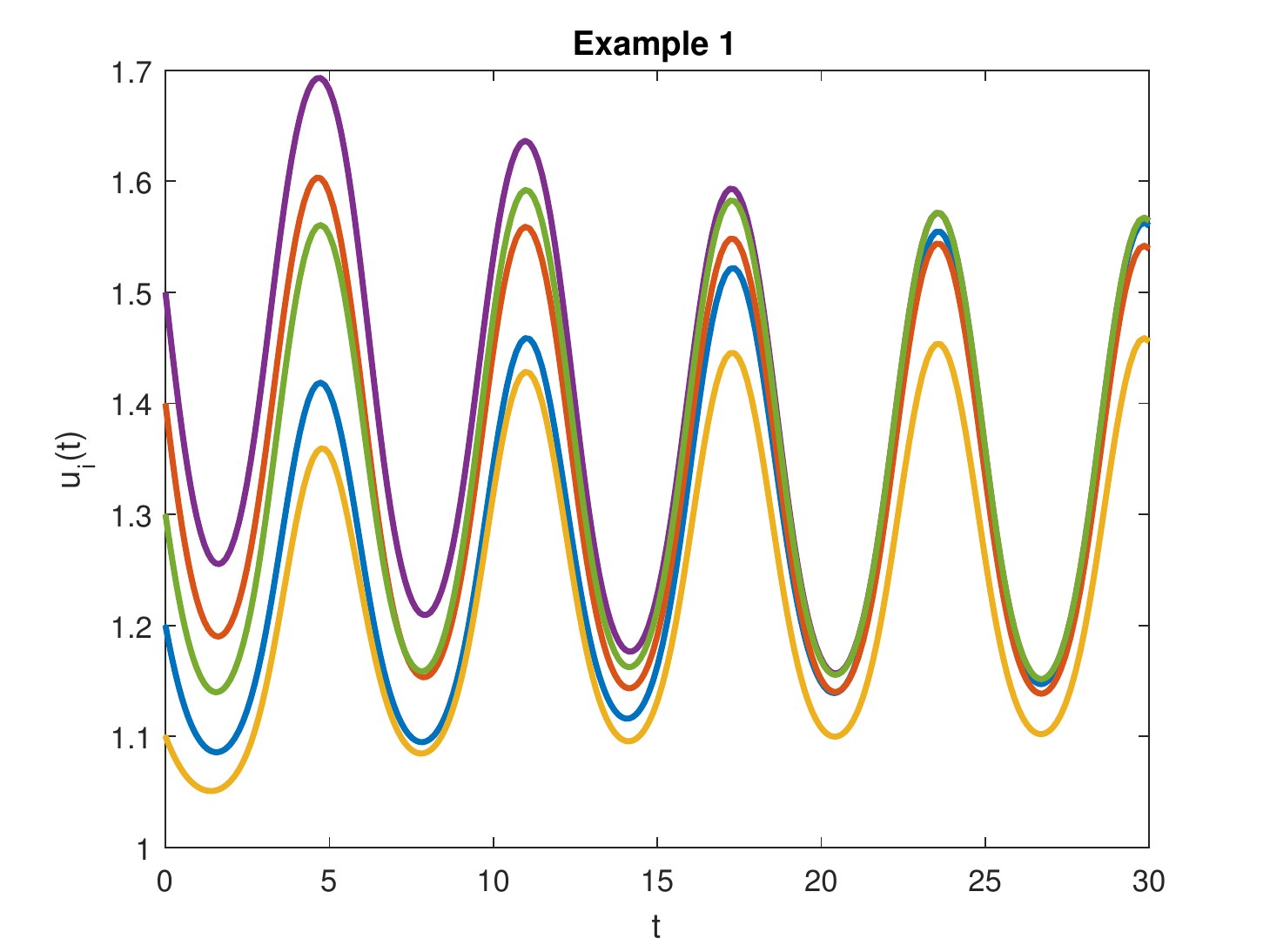}
\includegraphics[scale=0.4]{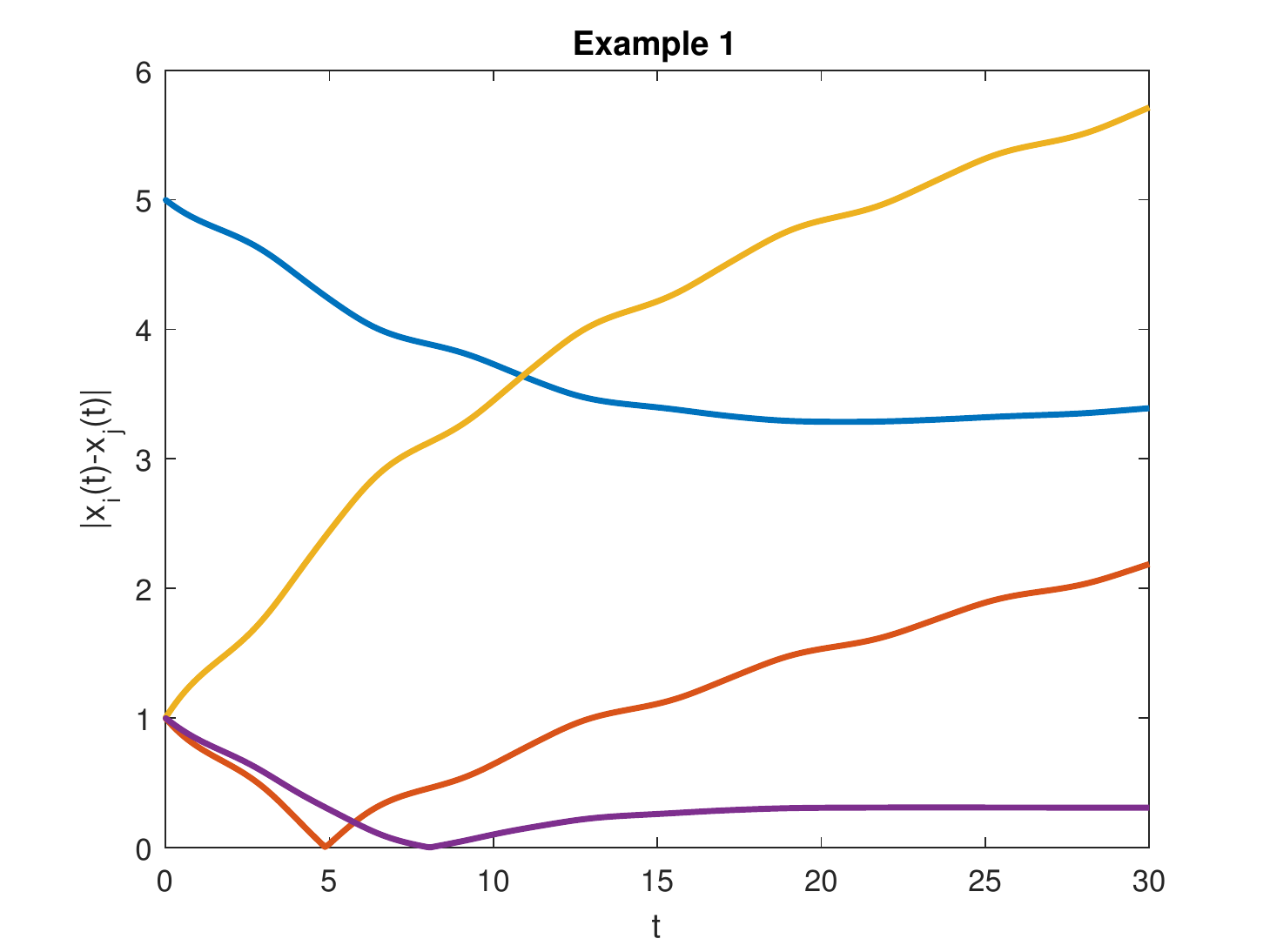}
\includegraphics[scale=0.4]{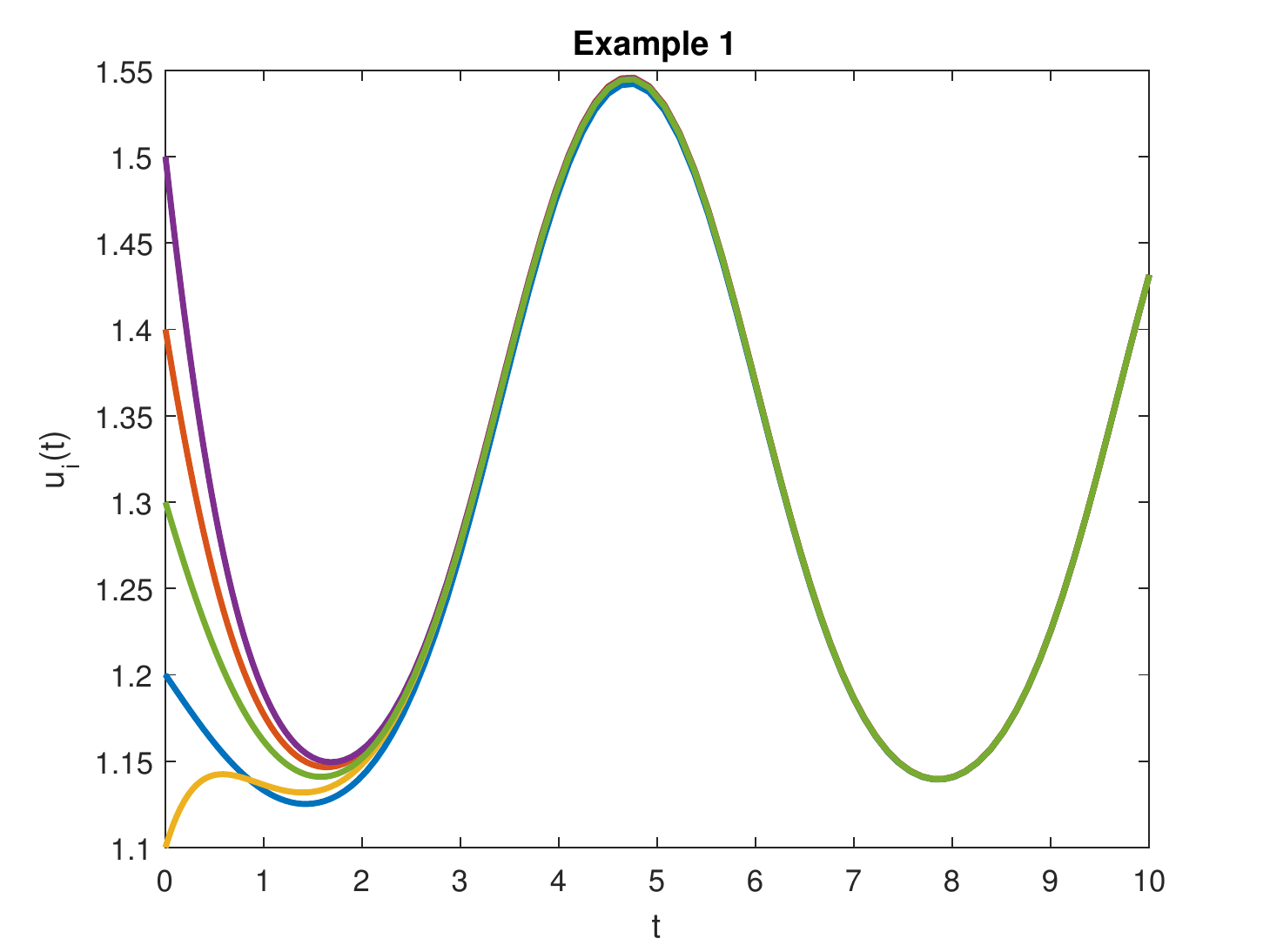}
\includegraphics[scale=0.4]{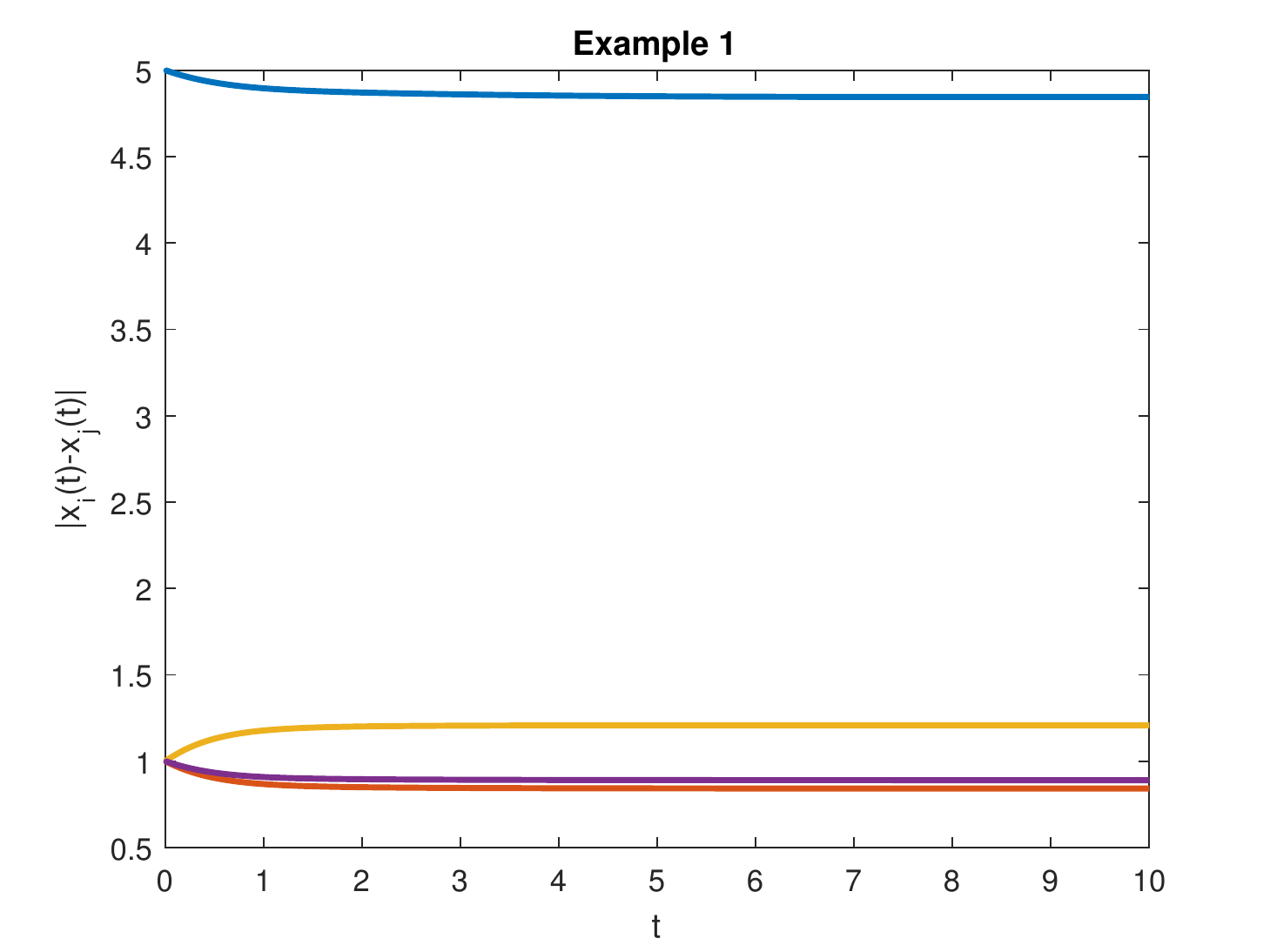}
\caption{Realization of $(x,v)$ of Example 1 with $\delta=10$, $\delta=4$ and $\delta=0.9$. The first column presents $v_i,~i\in [5]$ and the second column presents the differences $|x_i-x_j|,~i<j$.}\label{fig: example1_1}
\end{center}
\end{figure*}

\subsection{Example 2. Chaotic Flocking.}
The next example is on the problem of synchronizing chaotic oscillators. The nominal equation is chosen to be the Lorenz system
\begin{equation*}
%\begin{split}
g(z^{(1)},z^{(2)},z^{(3)})=\begin{bmatrix}
10(z^{(2)}-z^{(1)})\\
-z^{(2)}+z^{(1)}(28-z^{(3)})\\
-\frac{8}{3}z^{(3)}+z^{(1)}z^{(2)}
\end{bmatrix}.
%\dot z^{(1)}&=10(z^{(2)}-z^{(1)}) \\
%\dot z^{(2)}&=-z^{(2)}+z^{(1)}(28-z^{(3)})\\
%\dot z^{(3)}&=-\frac{8}{3}z^{(3)}+z^{(1)}z^{(2)}  
%\end{split}
\end{equation*} The solutions of \eqref{eq: nominal} converge for almost all initial values to a strange attractor \cite{sparrow1982lorenz}. This means that there is a $g$-invariant, convex subset $U\subset \mathbb R^3$ that includes the limit set. It can be verified that $U\subset \tilde{U}=[-17,17.5]\times [-22,24.5]\times [7,45]$. Furthermore, $\tilde U$ can be numerically verified to be $g$-invariant while it is clearly compact and convex. We can apply Theorem \ref{cor: main1} to \eqref{eq: model1} with $g$ as in this example and initial conditions in $v^0_i\in \tilde{U},~i\in [5]$. Then we can calculate an estimate on $K$ based on $U$ $$K\leq \max_{z\in \tilde{U}}\{0,|28-z^{(3)}|-1+|z^{(1)}|,|z^{(2)}|+|z^{(1)}|-\frac{8}{3}\}\approx 39.4\footnote{Clearly, better estimates on $K$ may be achieved when $\tilde{U}$ is substituted with a sharper estimate of $U$.}$$  We set $S(v^0)=S(x^0)=9$ and Theorem \ref{cor: main1} applies if one can find $d$ such that 
\begin{equation*}
\begin{split}
&9<\int_{9}^{d}\bigg(\frac{5w}{(r+2)^{\delta}}-39.4\bigg)\,dr ~~ \text{and} ~~ \frac{5w}{(d^*+2)^{\delta}}-39.4>0
\end{split}
\end{equation*} where $d^*$ when substitutes $d$ in the first condition achieves equality. Note that for any $\delta$ small enough one can always find $d^*$ such that
\begin{equation*}
9+39.4(d^*-9)> 39.4(d^*+2)^{\delta}\int_{9}^{d^*}\frac{dr}{(r+2)^{\delta}}
\end{equation*} This is a relation after solving for $w$ in the first condition (that is an equality for $d=d^*$) and substituting in the second condition. If for example we take $\delta=0.5$ then the last condition is satisfied with $d^*=11.67$ and $w\approx 150$. Figure \ref{fig:chflocking} depicts our results for strong $\delta=0.5$ and loose $\delta=7$ coupling, respectively. The rest of the parameters $w$ and initial configurations were kept identical. These two illustrative choices represent the case that conditions of Theorem \ref{cor: main1} are met and when the conditions are violated, respectively.
\begin{figure}[h]
\begin{center}
\includegraphics[scale=0.4]{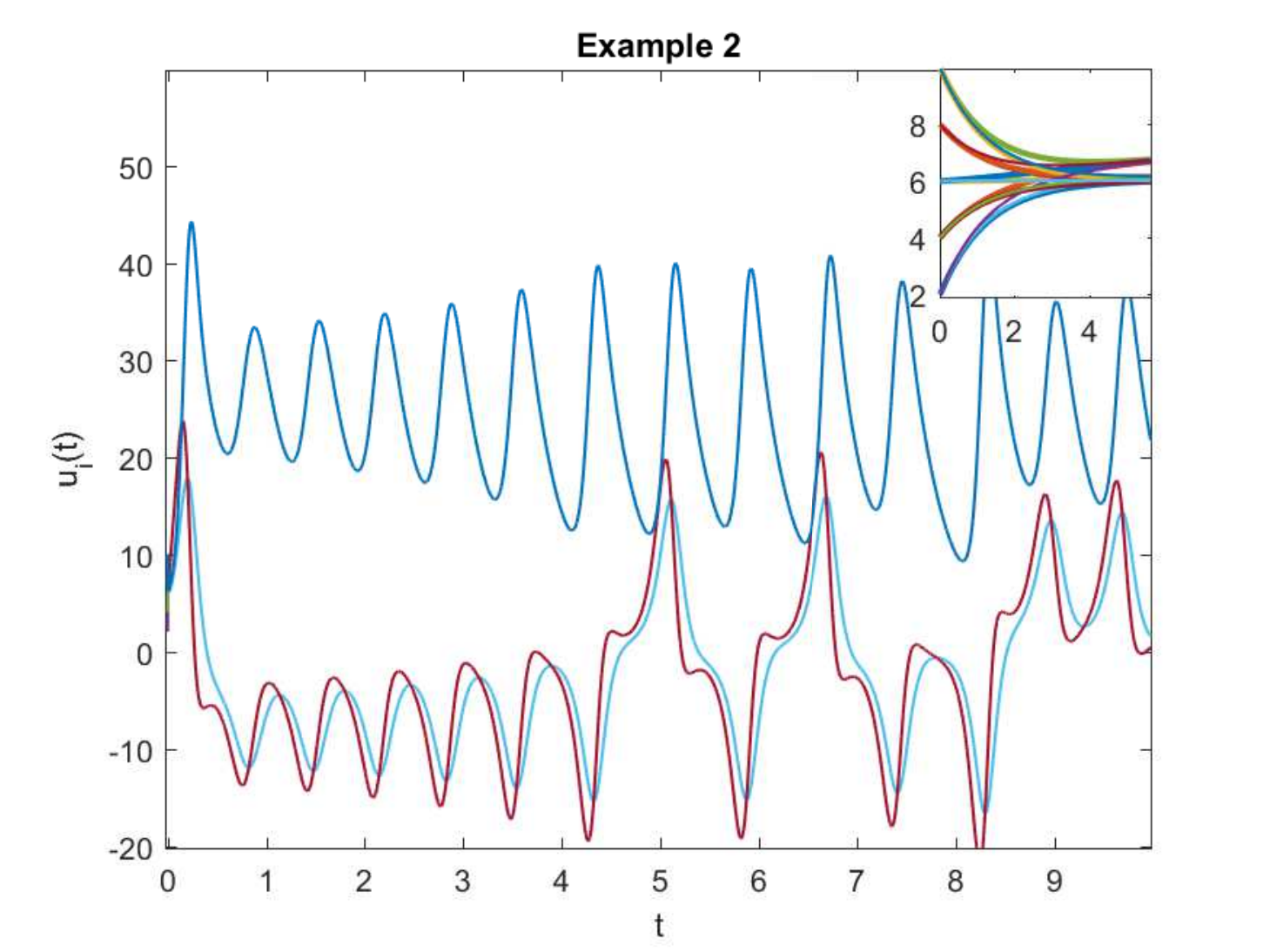}
\includegraphics[scale=0.4]{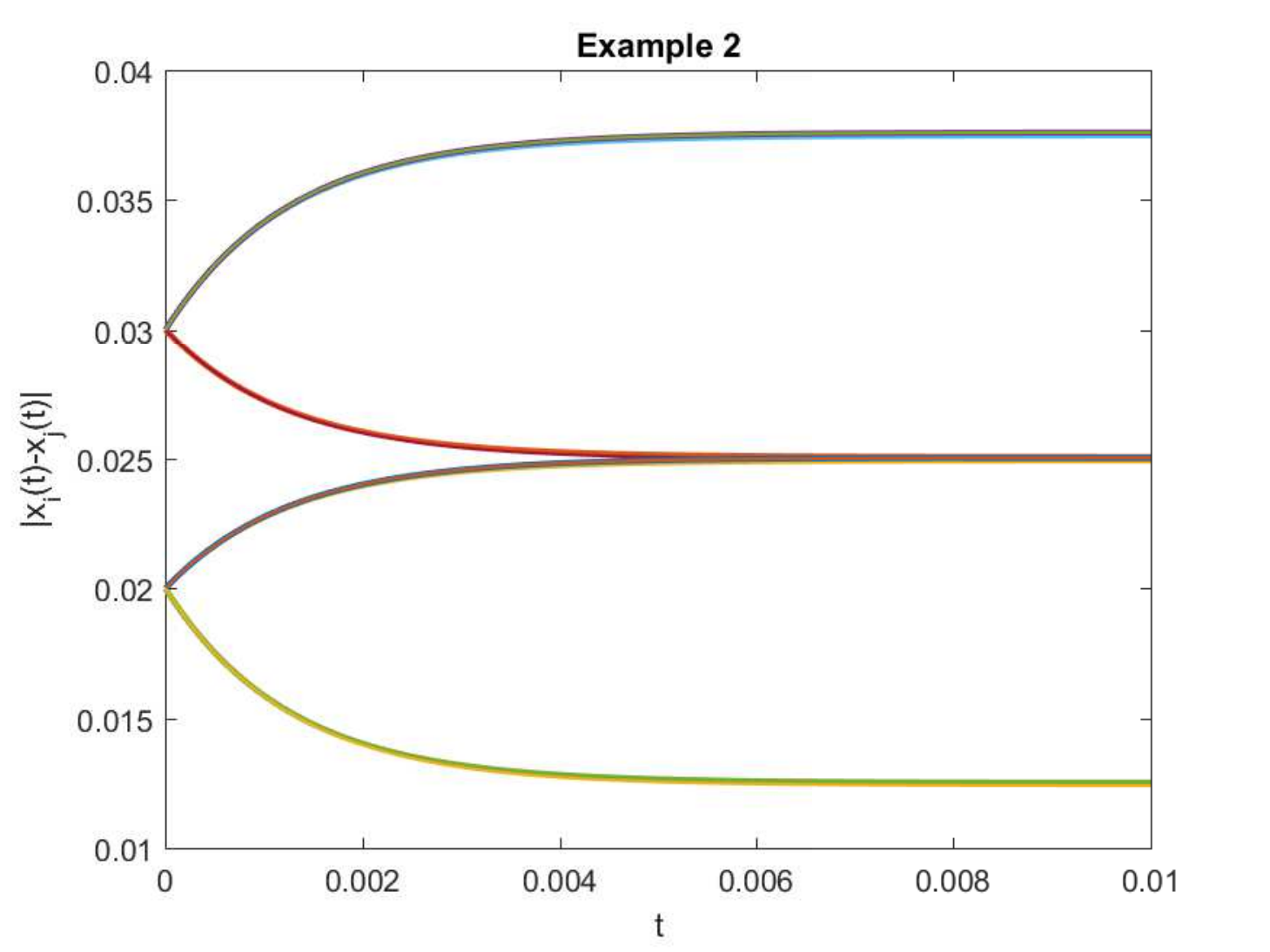}
\includegraphics[scale=0.4]{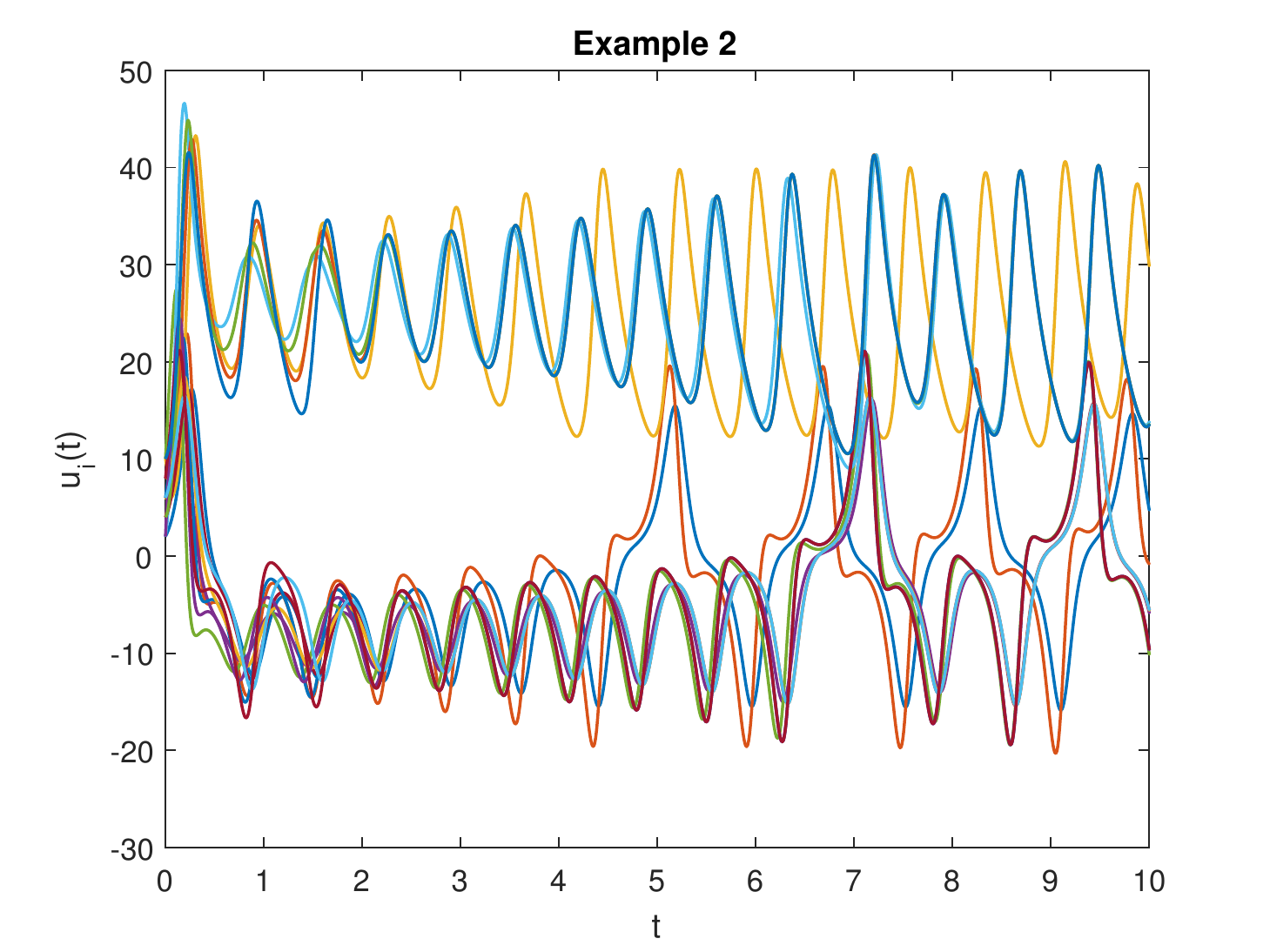}
\includegraphics[scale=0.4]{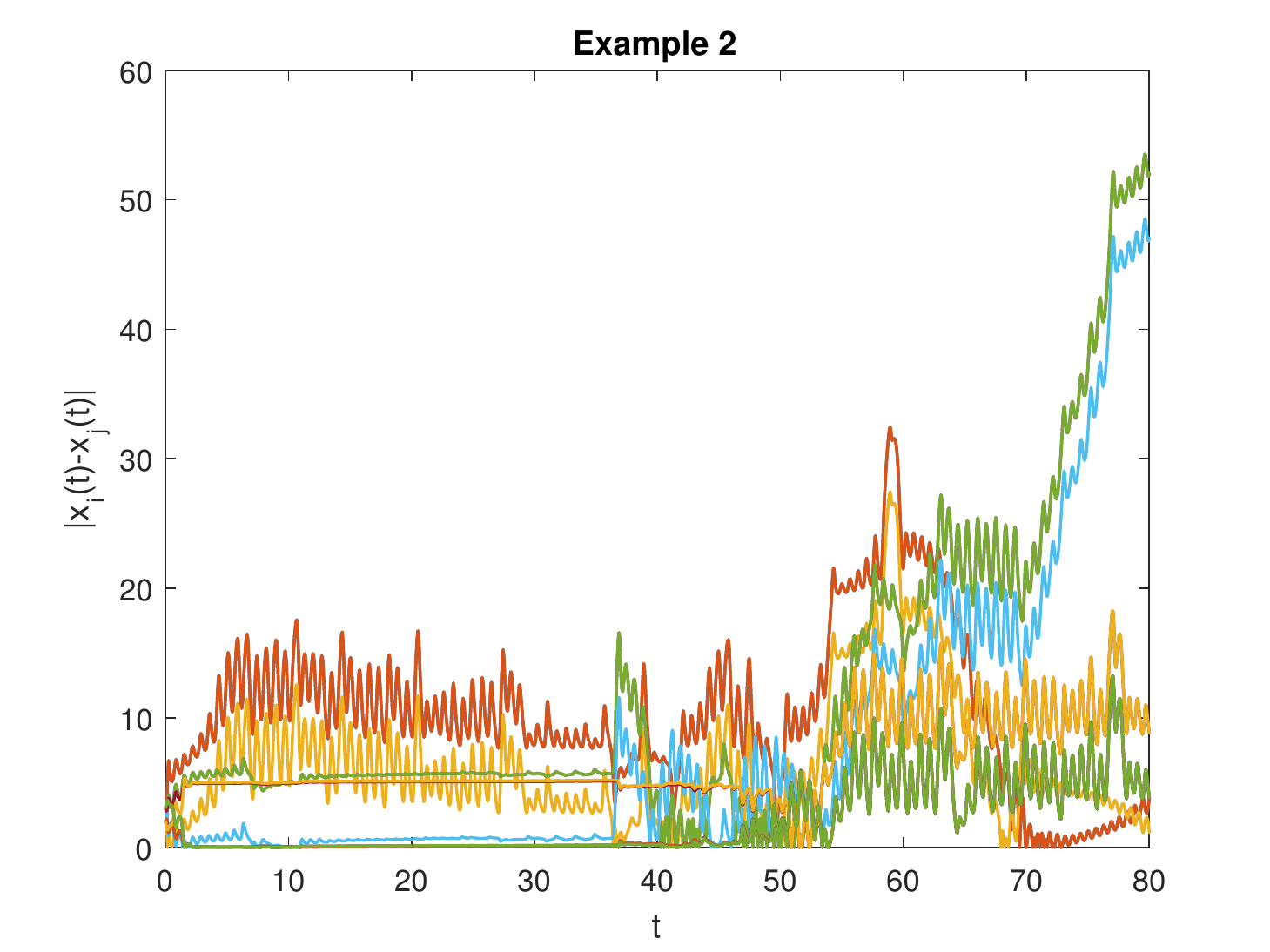}
\caption{Simulation Example 2. Synchronization of chaotic oscillators with $r=3$. In the first two figures we observe extremely strong synchronization under the theoretical sufficient conditions, $w=150$ and $\delta=0.5$. In the two figures that follow we present loose coupling with the same $w$ and $\delta=7$. We observe the of our network to syhnchronize.}\label{fig:chflocking}
\end{center}
\end{figure}

\subsection{Example 3. Collision Avoidance.} We conclude this section with an illustration of Theorem \ref{thm: main2}. We consider the network \eqref{eq: model2} with  $r=2$. The repelling functions functions are taken $$f_{ij}(r)=\frac{C_{ij}}{(r-r_0)^{\varphi}}$$ for numbers $C_{ij}$ arbitrarily chosen from $(1,2)$, $\varphi=1.5$ and $r_0=0.25$. Our initial data are $S(v^0)=6$ and $S(x^0)=4$. The first simulation runs with $\delta=1$ so that the condition of  Theorem \ref{thm: main2} is clearly satisfied and asymptotic flocking with collision avoidance is achieved. See Fig. \ref{fig:ca} for the simulation results. The repelling functions destabilize the agents' velocities in an uncontrollable manner. It then takes the strength of the coupling networks to determine the stability of flocking solutions. In the first case $\delta=1$ ensures the stability because it makes the right hand side of condition of Theorem \ref{thm: main2} unbounded.

\begin{figure*}[h]
\begin{center}
\includegraphics[scale=0.4]{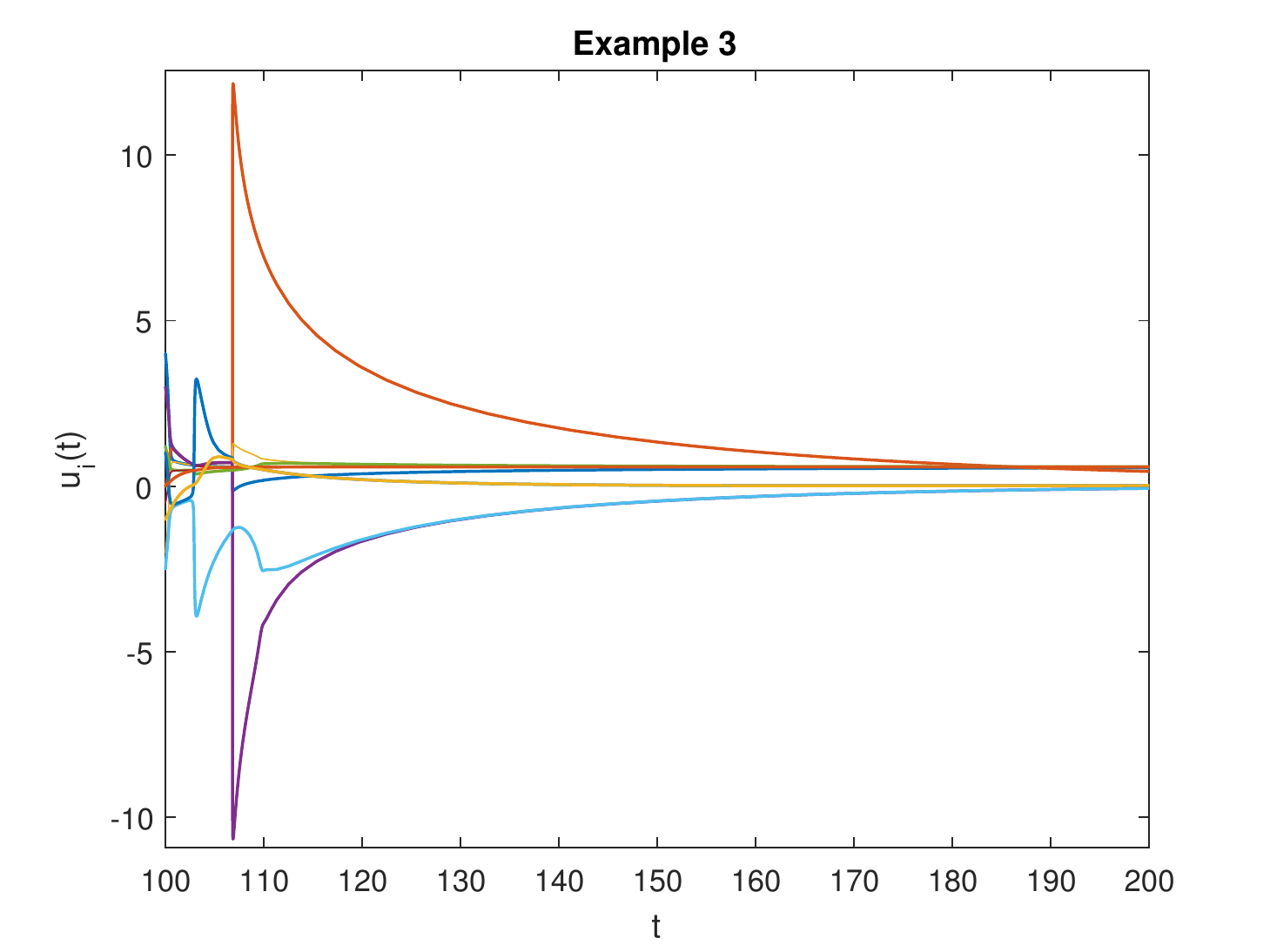}
\includegraphics[scale=0.4]{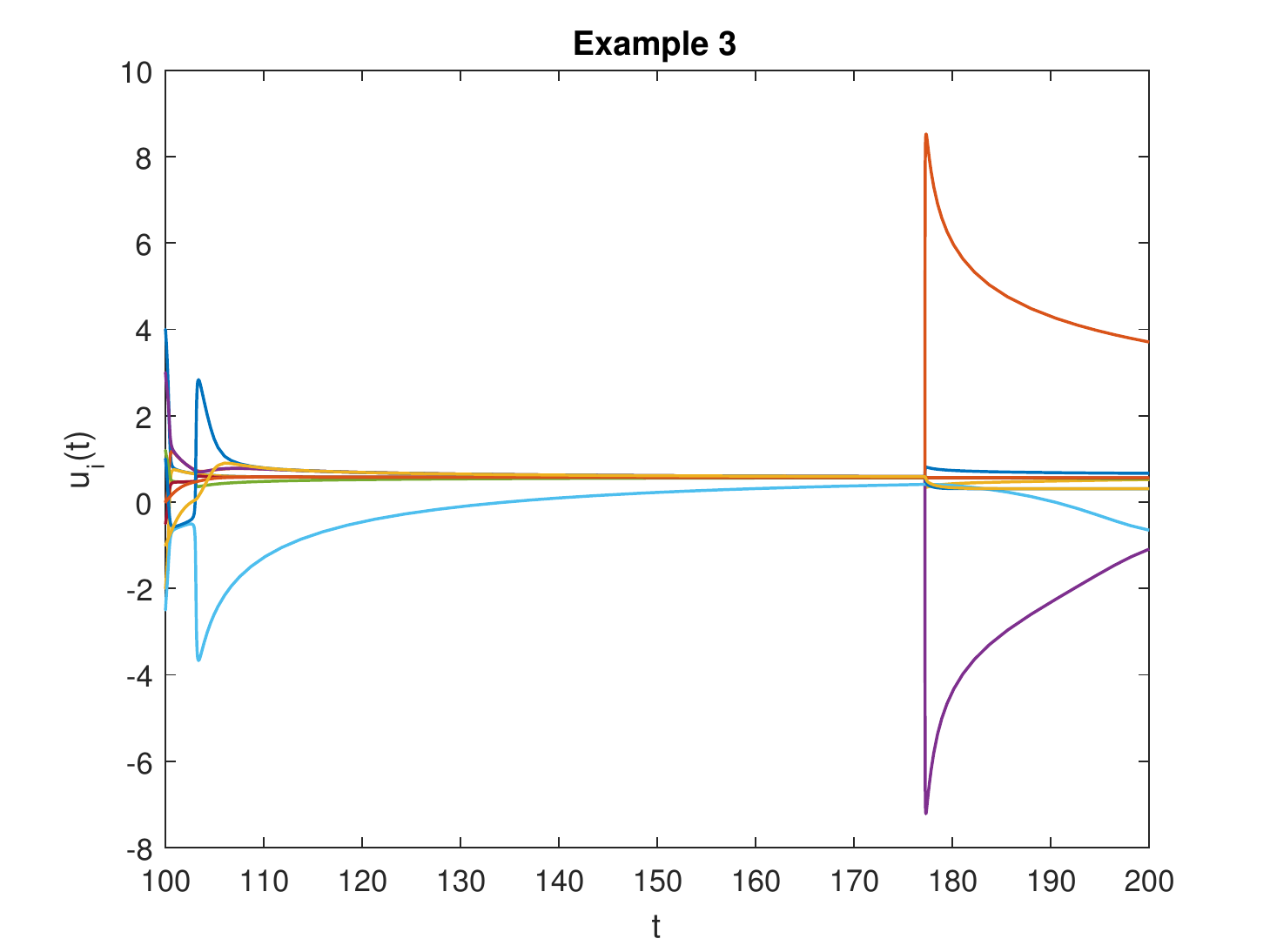}
\includegraphics[scale=0.4]{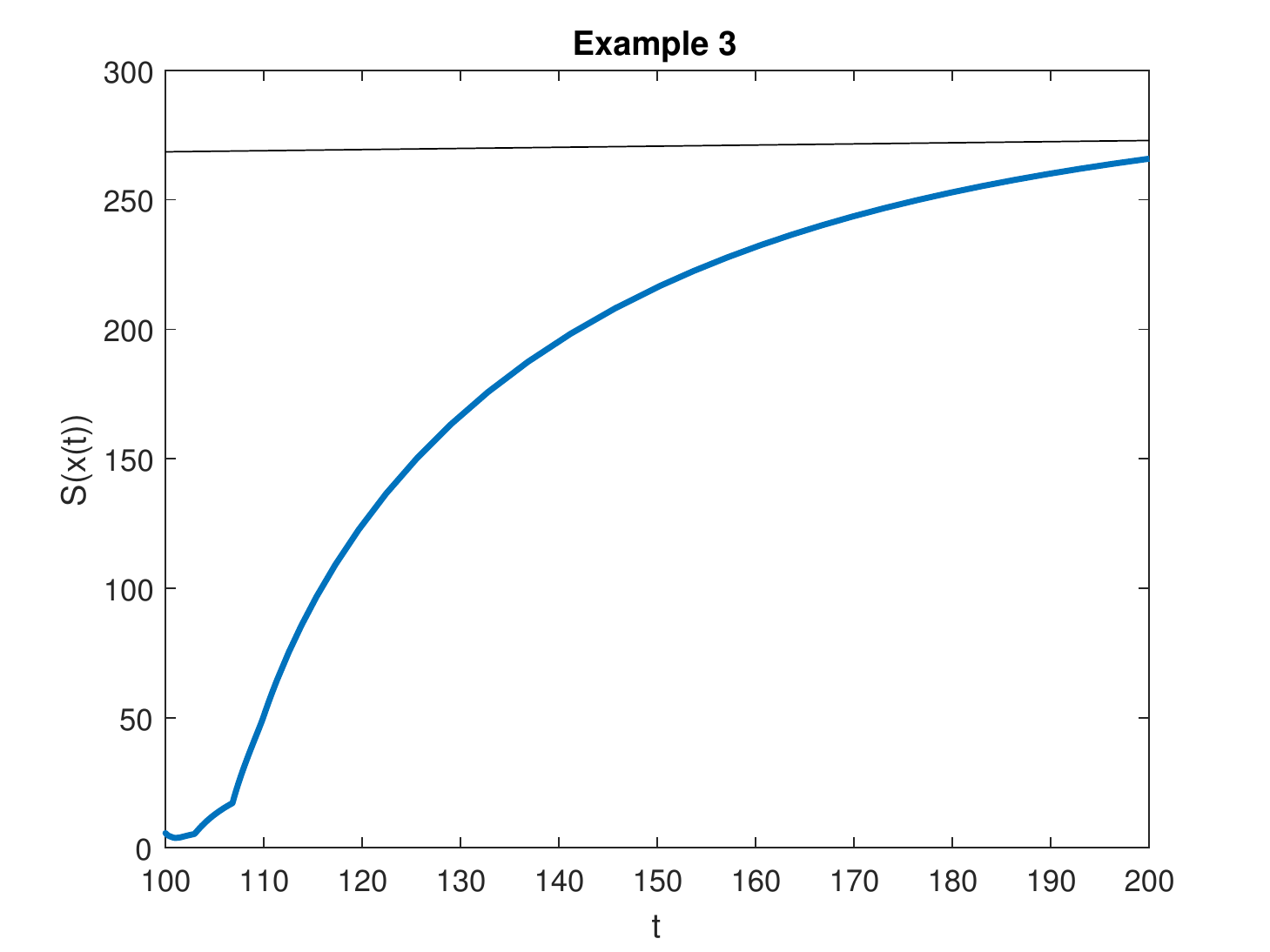}
\includegraphics[scale=0.4]{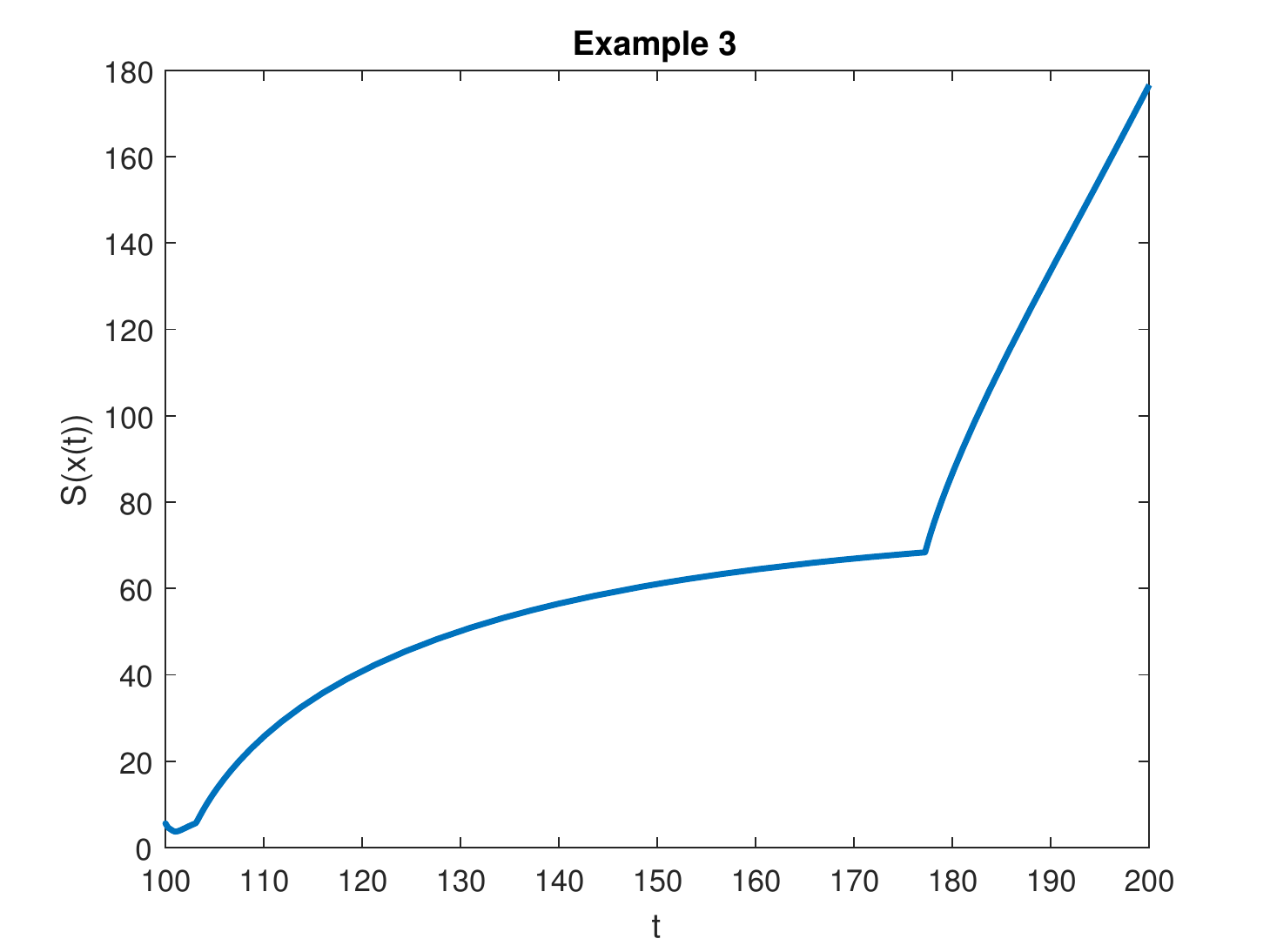}
\includegraphics[scale=0.4]{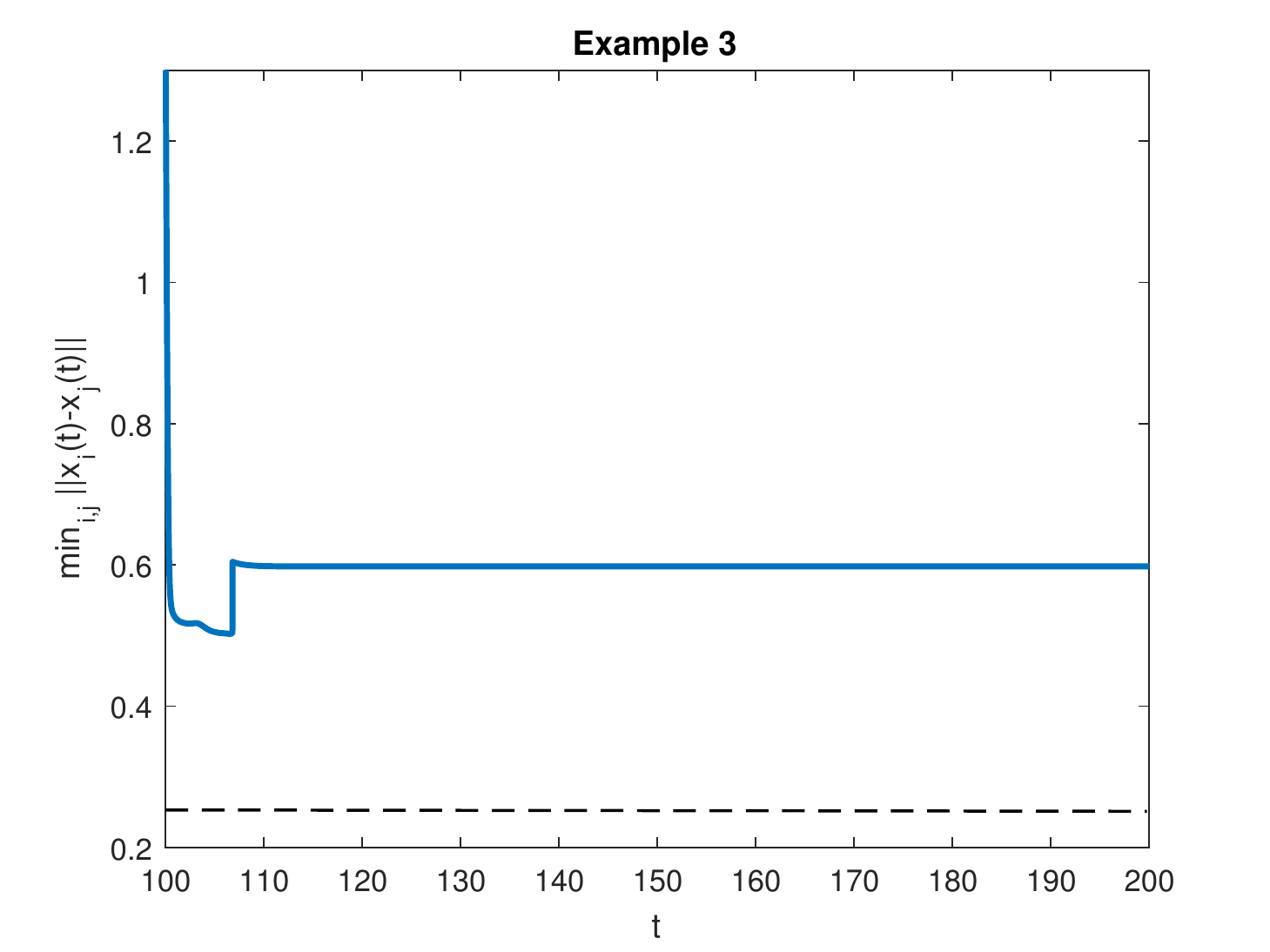}
\includegraphics[scale=0.4]{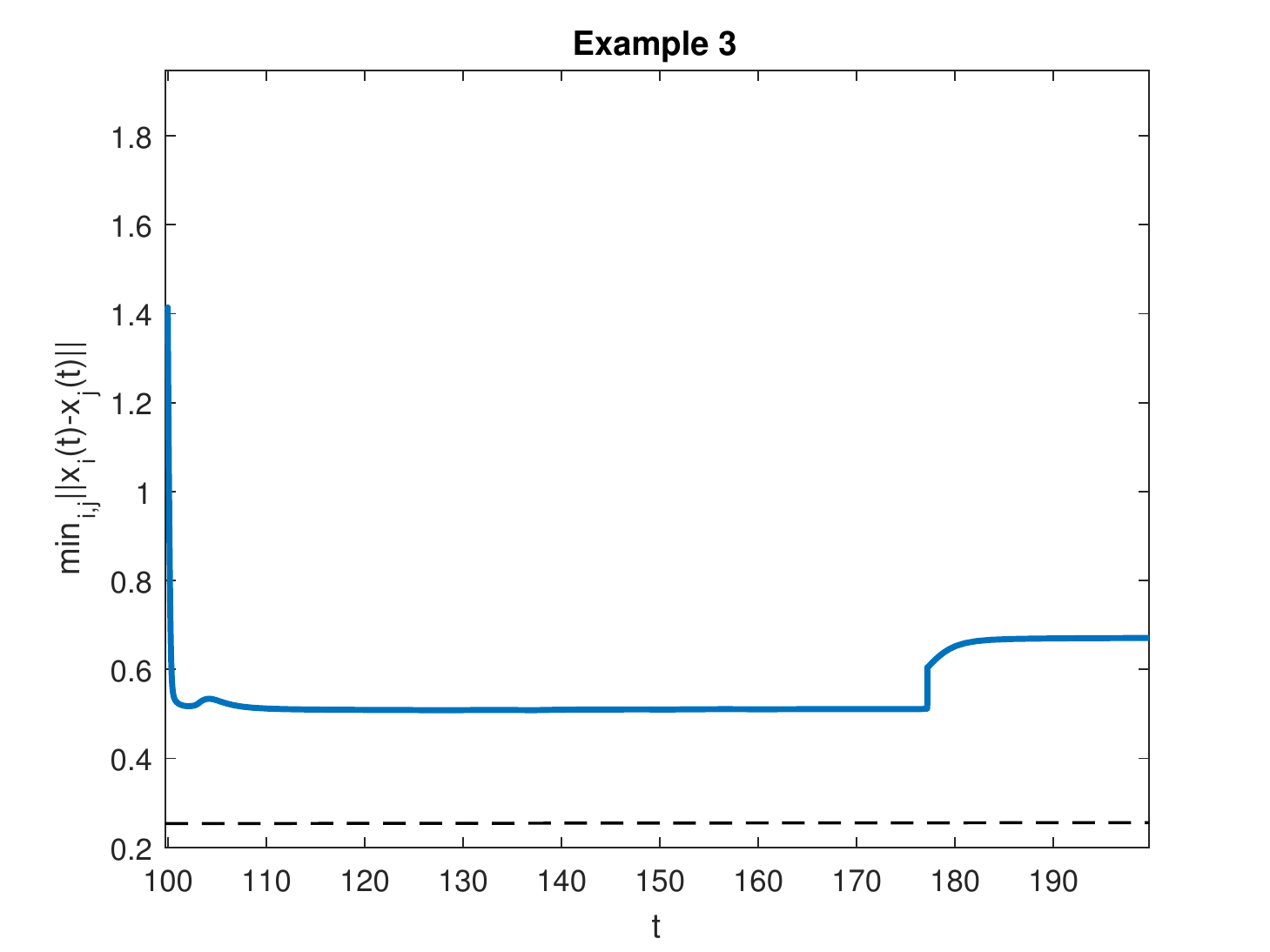}
\caption{Simulation Example 3. Flocking with collision avoidance. The first column is a strong coupling simulation of the velocity curves, the maximum distance and the minimum distance. The second column is a weak coupling simulation with the same quantities. While both simulations achieve collision avoidance, only the first one achieves flocking.}\label{fig:ca}
\end{center}
\end{figure*}

\section{Discussion} \label{sect: discussion} We addressed two variations of the classic non-linear flocking algorithm of Cucker-Smale type with asymmetric coupling rates. The striking similarity in the analysis of these two evidently different algorithms is clear: Both networks include the consensus-based stabilizing term and a potentially destabilizing term. The perturbations affect the stability of the network because it leads to fluctuations in the relative velocities $v_{i}-v_{j}$. This in turn results in weaker coupling strength the level of which can dissolve the network to fail to converge to consensus. This interplay is reflected in the initial configuration sufficient conditions that ensure convergence of the overall network to flocking behavior. 

As in Theorem \ref{thm: standardflocking}, the flocking to synchronization results in Theorem \ref{thm: main1} require convergence conditions to be parametrized by the initial setup states, the coupling strength in addition to the growth of the first derivative of the internal dynamics function. The argumentation is very similar for the case of the collision-free extension. We remark the striking similarity of Theorem \ref{thm: main2} to the results of \cite{Cucker_2011}. In fact Theorem \ref{thm: main2} may be considered as the the asymmetric alternative of \cite{Cucker_2011}.

%Such a modification essentially enabled us to study two seemingly different extensions of flocking networks. 

%
%Similar to our first model above, the result is a flocking condition that relates the initial positions and velocities of the agents to the system's coupling and repelling parameters. This is a sufficient condition on the balance between the stabilizing coupling term, the destabilizing repelling terms and the initial states.
%
%Theorem \ref{thm: main1}

Non-linearity imposes investigation of solutions that may not necessarily exist for all times. Theorem \ref{thm: main1} can accept such solutions and may be applied to networks where the solutions escape in finite time. Since this event is of little interest we readily turned to a result concerning solutions of the nominal network that exist for all times and derived claims for the coupled network.

It is yet to be investigated on the dynamics that can emerge with abstract nonlinear perturbations. Our strategy was to focus on reasonable forms of $b_{ij}$ that correspond to real-world mathematical models. In any case one could reasonably assume that positive valued $b_{ij}$'s in \eqref{eq: perturbcsflock} tend to stabilize the network towards a common constant equilibrium, while negative valued perturbation destabilize the network possibly towards  
more complex behaviors.
We conclude this paper on a brief discussion on extensions to our protocols that can be analyzed using similar tools, together with the corresponding limitations that emerge every time.

\subsection{Further Nonlinearities}
Models \eqref{eq: model1} and \eqref{eq: model2} can be easily extended to non-linearities in the spirit of \eqref{eq: non2} or \eqref{eq: non3}. In this case one would require additional assumptions on the type of these non-linearities. For instance note that \eqref{eq: non3} can be re-written as \begin{equation*}
\dot z_i=\sum_j w_{ij}(t,z)(z_i-z_j)
\end{equation*} where $w_{ij}(t,z)= \int_0^1 \frac{\partial}{\partial z}  w_{ij}(t,q z_i-(1-q)z_j)\,dq$. Then typical smoothness or uniformity assumptions on the non-linear $w_{ij}$ can make  a second order version possible to be analyzed with our framework, under appropriate modifications. Similar arguments can be adopted for couplings like \eqref{eq: non2}.

\subsection{The problem of dimensionality and different metrics}
For dynamics in one dimension, all metrics boil down to the absolute value making the analysis simple and elegant. In higher dimensional systems one must carefully chose the most appropriate metric that suits the geometry of the generated solutions. Our analysis dictates the option of $S(y)$ and any discrepancy with the solutions of the nominal dynamics may result in more conservative convergence conditions ( i.e. larger $K$). Attempts towards sharper sufficient conditions should probably allow some structure on the nominal dynamics $g$. In this case, the researcher can leverage on the geometry of the nominal solutions and their qualitative properties so as to come up with an appropriate contraction metric for the synchronization analysis.
\subsection{Connection Topologies}
A careful inspection of \eqref{eq: cc} reveals that the contraction coefficient that provides non-trivial estimates of the stochastic matrix $P$, must be associated with a graph that is rich in edges. In particular, it is asked that for every couple of two agents there is at least one agent that affects both of the aforementioned two. The connectivity we assumed in this work corresponds to a complete graph and clearly covers this case. The interest in networks of Cucker-Smale type lies primarily on the strength of the coupling weights and not on the topological structure. That is why they are usually assumed under all-to-all communication schemes. Looser connectivities are feasible for the synchronization network \eqref{eq: model1} but it involves more elaborate arguments, perhaps along the lines of \cite{SomBarecc15}. A collision free model without complete connectivity is, however, rather unreasonable. If two agents do not communicate somehow their positions are clearly prone to approach each other arbitrarily close as they will simply not notice each other. In either case one should expect even more conservative stability conditions. The weaker the connectivity over the network, the smaller the rate estimates become. In either case, the problem of nonlinear synchronization under weaker connectivity regimes remains of interest and will be considered in the future.

\appendix
\begin{proof}[Proof of Lemma \ref{lem: difineq}] Let $(x(t), v(t)),~t\in [t_0,T)$ the maximal solution of \eqref{eq: model1}. Fix $i,i^{'}\in [n]$, $l\in [r]$ and take $m=\bar{n}\bar{w}+K$. %the function $m(t)\in C^0([t_0,T),\mathbb R)$ as
%\begin{equation*}
%m(t)=\bar{n}\bar{w}-\int_{0}^{1}\frac{\partial }{\partial z^{(l)}}g^{(l)}(t,qv_i^{(l)}(t)+(1-q)v_{i'}^{(l)}(t))\,dq
%\end{equation*} 
where $\bar{n}>n+1$, $\bar{w}$ as in Assumption \ref{assum: f}. From the second line of \eqref{eq: model1} we have 
\begin{equation*}\begin{split}
&\dot v_j^{(l)}=-m v_j^{(l)}+\big(m-d_j\big)v_j^{(l)}+\sum_{k}w_{jk}v_j^{(l)}+g^{(l)}(t,v_j)
\end{split}
\end{equation*} for $j=i,i^{'}$, $w_{jk}=w_{jk}\big(t,x(t)\big)$, $d_j=\sum_{k\neq j} w_{jk}$. Set  $\Delta_{i,i^{'}}(t)=e^{-m(t-t_0)}\frac{d}{dt}\big(e^{m(t-t_0)}[v_i^{(l)}(t)-v_{i^{'}}^{(l)}(t)]\big)$ and substituting the equations on $i$ and $i^{'}$ we obtain
\begin{equation*}
\Delta_{i,i^{'}}(t)=\sum_{k}(\tilde w_{ik}-\tilde w_{i^{'}k})v_k^{(l)}+\sum_{h\neq l}\bigg[\int_0^{1}\frac{\partial}{\partial z^{(h)}}g^{(l)}(t,q v_i+(1-q)v_{i'})\,dq\bigg] \big(v_{i}^{(h)}-v_{i^{'}}^{(h)}\big)
\end{equation*} where  $\tilde w_{ij}=\tilde w_{ij}(t,l)$ defined as
\begin{equation*}
\tilde w_{ij}=\begin{cases}
m-d_i+c_{i,i^{'}}, & j=i \\
w_{ij},  & j\neq i.
\end{cases}
\end{equation*} for $c_{i,i^{'}}=\int_{0}^{1}\frac{\partial }{\partial z^{(l)}}g^{(l)}(t,q v_i(t)+(1-q)v_{i^{'}}(t))\,dq$. Similarly for $\tilde{w}_{i^{'}j}$. Take $a_{j}=(\tilde w_{ij}-\tilde w_{i^{'}j})$ and observe that \begin{equation*}\begin{split}
\sum_j a_j&=\sum_{j\neq i} w_{ij}+w_{ii}-\sum_{j\neq i^{'}} w_{i^{'}j}-w_{i^{'}i^{'}}=d_i+m-d_i+c_{i,i^{'}}-d_{i^{'}}-m+d_{i^{'}}-c_{i,i^{'}}=0.
\end{split}
\end{equation*}
The index $j$ for which $a_{j}>0$ is denoted by $j^+$ and the index for which $a_{j}\leq 0$ is denoted by $j^-$. Set $\theta=\theta(t)$ \begin{equation*}\begin{split}
\theta&=\sum_{j^+}a_{j^+}=\sum_{j^+}|a_{j^+}|=-\sum_{j^-}a_{j^-}=\sum_{j^-}|a_{j^-}|=\frac{1}{2}\sum_{j}|a_j|=\frac{1}{2}\sum_{j}|\tilde{w}_{ij}-\tilde w_{i'j}|
\end{split}
\end{equation*} Then for $t\in [t_0,T)$  \begin{equation*}
\begin{split}
\Delta_{i,i^{'}}(t)&=\theta \bigg(\frac{\sum_{j^+}|a_{j^+}|v_{j^+}}{\theta}-\frac{\sum_{j^-}|a_{j^-}|v_{j^-}}{\theta}\bigg)+\sum_{h\neq l}\bigg[\int_0^{1}\frac{\partial}{\partial z^{(h)}}g^{(l)}(t,q v_i+(1-q)v_{i'})\,dq\bigg] \big(v_{i}^{(h)}-v_{i^{'}}^{(h)}\big)\\
%&\leq \theta S(v)+\sum_{h\neq l}\bigg[\int_0^{1}\frac{\partial}{\partial z^{(r)}}%g^{(l)}(t,q v_i+(1-q)v_{i'})\,dq\bigg]\big(v_{i}^{(r)}-v_{i^{'}}^{(r)}\big)\\
&\leq \bigg(\theta+\sum_{h\neq l}\bigg|\int_0^{1}\frac{\partial}{\partial z^{(h)}}g^{(l)}(t,q v_i+(1-q)v_{i'})\,dq\bigg|\bigg) S(v)
\end{split}
\end{equation*} But from the identity $|x-y|=x+y-2\min\{x,y\}$ and the particular $m$ we chose, it can be deduced that 
\begin{equation*}
\begin{split}
\theta=m+c_{i,i^{'}}-\sum_{k\neq i,i'}\min\{w_{ik}(t),w_{i'k}(t)\}
\end{split}
\end{equation*}
So that $$ \Delta_{i,i^{'}}(t)\leq  \bigg( m+K-\sum_{k\neq i,i'}\min\{w_{ik}(t),w_{i'k}(t)\} \bigg) S(v).$$ Finally, for $i,i'$ and $l$ that maximize $|v_i^{(l)}(t)-v_{i'}^{(l)}(t)|$, we have \begin{equation*}\begin{split}
\frac{d}{dt}S(v(t))&=\frac{d}{dt}\big[e^{-m(t-t_0)}S\big(e^{m(t-t_0)} v(t)\big)\big]=-m S(v(t))+\Delta_{i,i^{'}}(t) \leq  \big[K-n\psi\big(S(x(t))\big)\big]S\big(v(t)\big).
\end{split}
\end{equation*} that is true for $t\in [t_0,T)$.
\end{proof}
%\vspace{-20pt}
\begin{proof}[Proof of Lemma \ref{lem: difineq2} ] Consider the maximal solution $\big(x(t),v(t)\big)$ on $[t_0,T)$. In any $l\in [r]$ dimension, $v_i^{(l)}$ satisfies $$\dot v_i=-m v_i+(m-d_i)v_i+\sum_{j}(w_{ij}+b_{ij})v_j,$$ where $w_{ij}=w_{ij}(t,x(t))$, $b_{ij}=b_{ij}(t,x(t),v(t))$ as in \eqref{eq: bij}, $d_i=d_i(t)=\sum_{j}(w_{ij}+b_{ij})$ and $m=m(t)$ to be determined below. Note that we removed the $l$ dependency because the dynamics are identical in all $[r]$ dimensions. We set $\phi(t,t_0)=\exp\{\int_{t_0}^{t}m(s)\,ds\}$. Following the proof of Theorem \ref{thm: main1}, $\Delta_{i,i^{'}}(t):=\phi^{-1}(t,t_0)\frac{d}{dt}\big[\phi(t,t_0)|v_i(t)-v_{i^{'}}(t))|\big]$, yields $$\Delta_{i,i^{'}}(t)=\sum_j \alpha_j v_j$$ for 
\begin{equation*}
\begin{split}
\alpha_j=\begin{cases}
(w_{ij}-w_{i'j})+(b_{ij}-b_{i'j}), & j\neq i,i^{'}\\
m-d_i-w_{ii'}-b_{i'i},&  j=i,\\
-m+d_{i'}+w_{i'i}+b_{ii'},&  j=i^{'}
\end{cases}
\end{split}
\end{equation*} Let $j'$ denote the indices $j$ for which $\alpha_{j}>0$ and $j''$ the indices for which $\alpha_{j}\leq 0$. Then $$\Delta_{i,i^{'}}(t)=\sum_{j'}|\alpha_{j'}|v_{j'}-\sum_{j''}|\alpha_{j''}|v_{j''}.$$ Again, it holds that $\sum_{j}\alpha_j\equiv 0$ so we take $\theta_{ii^{'}}=\theta_{ii^{'}}(t)$ to be the sum of the positive $\alpha_j$.
\begin{equation*}
\begin{split}
\theta_{ii^{'}}&=\sum_{j'}|\alpha_{j'}|=-\sum_{j''}\alpha_{j''}=\sum_{j''}|\alpha_{j''}|=\frac{1}{2}\sum_{j}|\alpha_j|\\
&\leq \frac{1}{2}\sum_{j}|w_{ij}-w_{i'j}|+\frac{1}{2}\sum_{j}|b_{ij}-b_{i'j}|\\
%=n-\sum_{j}\min\{w_{ij}+b_{ij},w_{i'j}+b_{i'j}\}\\
%-\sum_{j}\min\{b_{ij},b_{i'j}\}\\
%\frac{1}{2}\sum_{j}|(w_{ij}+b_{ij})-(w_{i'j}-b_{i'j})|\\
&=m-\sum_{j\neq i}\min\{w_{ij},w_{i'j}\}-\sum_{j\neq i}\min\{b_{ij},b_{i'j}\}
\end{split} 
\end{equation*} in view of the identity $2\min\{x,y\}=x+y-|x-y|$. 
\begin{equation*}\begin{split}
&\Delta_{i i^{'}}(t)=\theta\bigg(\frac{\sum_{j'}|w_{j'}|v_{j'}}{\sum_{j'}|w_{j'}|}-\frac{\sum_{j''}|w_{j''}|v_{j''}}{\sum_{j''}|w_{j''}|}\bigg)
\end{split}
\end{equation*} 
Invoking the estimate on $\theta$
\begin{equation*}
\begin{split}
&\Delta_{i i^{'}}(t)\leq \big(m-\sum_{j\neq i}\min\{w_{ij},w_{i'j}\}-\sum_{j\neq i}\min\{b_{ij},b_{i'j}\}\big) S(v)=\big(m-\rho_{i,i^{'}}\big) S(v)-\Gamma_{i,i^{'}}
\end{split}
\end{equation*}
where $\rho_{i,i^{'}}$, $\Gamma_{i,i^{'}}$ as in the statement of the Lemma. Observe that $\frac{d}{dt}|v_i(t)-v_{i^{'}}(t)|=-m|v_{i}-v_{i^{'}}(t)|+\Delta_{i i^{'}}(t) $ and substitute the estimate of $\Delta_{i,i^{'}}(t)$ to conclude.
\end{proof}

%\subsection{On the strength of the results}
\bibliographystyle{plain}    
\bibliography{bibliography} 
\end{document}